\documentclass[11pt]{article} 

\usepackage{graphicx} 
\usepackage{setspace}

\usepackage{amsmath,stix}
\usepackage{amsfonts}

\usepackage{subcaption}

\usepackage{tikz}
\usepackage{pgfplots}

\usepackage{enumitem}

\usepackage{hyperref}

\usepackage{pifont}

\usepackage{amsthm}
\newtheorem{theorem}{Theorem}[section]
\newtheorem{lemma}[theorem]{Lemma}
\newtheorem{corollary}[theorem]{Corollary}
\newtheorem{proposition}[theorem]{Proposition}
\newtheorem{definition}[theorem]{Definition}
\newtheorem{example}[theorem]{Example}
\newtheorem{remark}[theorem]{Remark}
\newtheorem{conjecture}[theorem]{Conjecture}
\newtheorem{question}[theorem]{Question}

\numberwithin{equation}{section}
\counterwithin{figure}{section}

\def\XXint#1#2#3{{\setbox0=\hbox{$#1{#2#3}{\int}$}
\vcenter{\hbox{$#2#3$}}\kern-.5\wd0}}

\DeclareMathOperator{\spt}{spt}
\DeclareMathOperator{\proj}{proj}
\DeclareMathOperator{\sgn}{sgn}
\DeclareMathOperator{\Id}{Id}

\begin{document}

\title{On the asymptotic behavior of the Repulsive \\ Pressureless Euler-Poisson System}
\author{Nicholas Biglin, Columbia University, nb3101@columbia.edu \\ Joseph Crachiola, Wayne State University, ho9872@wayne.edu\\ Jack Curtis, Worcester Polytechnic Institute, jtcurtis@wpi.edu \\ Thomas Kunz, University of Minnesota, kunz0119@umn.edu \\Omkar Maralappanavar, University of Connecticut, omkar.maralappanavar@uconn.edu\\ Adrian Tudorascu, West Virginia University, adtudorascu@mail.wvu.edu}

\maketitle
\begin{abstract}
    The main objective of this paper is a study of the asymptotic behavior of distributional solutions to the one-dimensional repulsive pressureless Euler-Poisson system.
    The system is a model for the dynamics of a mass distribution evolving on $\mathbb R$ whose masses exert outward forces on one another. A discrete (describing the evolution of finitely many particles) solution is called \textit{sticky} if, upon collision, particles stick together and move as one for all subsequent time, according to the conservation of mass and momentum principles.
    We prove results on the total energy (Hamiltonian) of the system and demonstrate the existence and uniqueness of so-called ``perfect'' states, where the Hamiltonian is constant over all time and the solution converges to equilibrium,  a single stationary particle. We provide a necessary and a sufficient condition for finite-time collapse, and present a quadratic envelope within which a solution must remain in order to collapse. We demonstrate various (counter)examples that illustrate the unique behavior of the repulsive scheme with the sticky condition, analytically and with a computer simulation. 
\end{abstract}

\doublespacing
\section{Introduction}
    \subsection{Overview}
    Consider $-\infty<y^0_1<\ldots<y^0_n<\infty$ and assign a positive mass $m_i$ to each $y_i^0$, $i=1,\ldots,n$ such that $m_1+\ldots+m_n=1$. Furthermore, assign a velocity $v^0_i\in\mathbb{R}$ to each ``particle'' $p_i$ (located initially at $y^0_i$ and of mass $m_i$) and then let the ensemble evolve according to the conservation of mass and momentum principles, and under a repulsive interaction potential. We assume each particle moves from its initial location with its initial velocity and under an acceleration equal to the difference between the total mass to the right and the total mass to the left. At collision of a number of particles, they stick together to form a particle of mass equal to the sum of the masses of the colliding particles. The velocity of the newly formed particle is given by the conservation of momentum. It turns out that if one denotes the empirical probability distributions and their velocities thus formed, the pair $(\rho,v)$ solves the system
    \begin{equation}\label{PEP}
        \begin{cases}\partial_t \rho +   \partial_y(\rho v)=0 \\
        \partial_t(\rho v)+\partial_y(\rho v^2)=\frac12 (\text{sgn}*\rho)\rho~~,
    \end{cases}
    \end{equation}
    in the sense of distributions on $(0,\infty)\times\mathbb{R}$. 
    Originally, Zeldovich \cite{Zeldovich} proposed the pressureless Euler system (the right hand side of the momentum equation above is zero, i.e. there is no long range interaction between particles so they travel at constant velocity between collisions) as a toy model for formation of large structures in the universe by accretion of matter. Its mathematical analysis has been surprisingly incremental, as the natural space for solutions is not a functional one but one that consists of general probability measures. In the case of pressureless Euler, the earliest efforts \cite{ERS}, \cite{BrenGren} contended with the basic question of existence of solutions in the case where the initial distribution is more general than a convex combination of Dirac masses. While neither of these pioneering works covered the most general case, the variational formulation in \cite{ERS} is reminiscent of the projection formula from \cite{Natile-Savare}, \cite{BLPTW} while the approach via scalar conservation laws from \cite{BrenGren} was later used in \cite{Tudorascu-2008} to extend the existence results to general initial data and the case where the right-hand side of the momentum equation is an interaction term that includes both the attractive and repulsive pressureless Euler-Poisson systems. In \cite{Tudorascu2015} existence results are extended to the case where a viscosity term of a particular form is added to the momentum equation. 
     For the pressureless Euler or attractive presureless Euler-Poisson systems the time evolution is explicitly described in terms of a projection onto the convex cone of square integrable, nondecreasing functions on the unit interval \cite{Natile-Savare}, \cite{BLPTW}. This, together with earlier results on uniqueness of solutions satisfying certain entropy-like conditions \cite{Huang-Wang}, allows us to define a class of {\it sticky} solutions, which contains all the discrete sticky solutions and is stable with respect to the initial data. Furthermore, sharp results on convergence to equilibrium were obtained in \cite{BLPTW} as a result of that explicit projection formula. Prior to that, Hynd and Tudorascu \cite{Hynd-Tudorascu} proved that compactly supported sticky particles solutions to Pressureless Euler converge to an asymptotic equilibrium; in \cite{Tudorascu-Wassink} it was proved that, as opposed to the attractive pressureless Euler-Poisson system, the convergence to equilibrium can be arbitrarily slow. Existence and uniqueness of sticky particles solutions for  the pressureless Euler system in a bounded domain with sticky boundary conditions was established in \cite{Tudorascu2023}.

     It stands to reason that in the attractive Euler-Poisson case all sticky particles solutions converge to an equilibrium consisting of a single stationary Dirac \cite{BLPTW}. In the absence of the attraction force, the interaction-less case, i.e. pressureless Euler, predicates the existence of an asymptotic equilibrium solely on the interplay between the initial positions and velocities \cite{Hynd-Tudorascu}. Independent of the physical intuition, the ``explicit'' formula for sticky particles solutions is available in both cases, so sharp results on the asymptotic behavior of solutions are derived directly from that; it is unnecessary to approximate the solutions with discrete ones in order to draw conclusions in the limit.
     It was already known that such a projection formula does not hold in the repulsive case \cite{Gangboetal} and we construct explicit examples where it fails instantaneously. We also show that even the class of discrete solutions is not stable with respect to the initial data, as non-sticky discrete solutions occur as limits of sticky discrete solutions. It was discussed in the literature \cite{Gangboetal} that the repulsive interaction is antagonistic to the stickiness principle (whereas the opposite is true in the attractive case), so it is unsurprising that one is forced to extend the class of distributional solutions to one we dubbed ``generalized sticky particles solutions'', which is simply defined as the closure of the class of sticky discrete solutions. 

     We note here that related works on biological aggregation deal with short range repulsive and long-range attractive interactions, variational solutions for, spatial confinement and the asymptotic behavior of the resulting dynamics. We mention only a few here: \cite{Balague}, \cite{Carrilloetal}, \cite{Fet-Huang}, \cite{Frank}. 
     
     We study the total energy (Hamiltonian) of the system and in the finitely many particles case we demonstrate the existence and uniqueness of so-called ``perfect'' states, where the Hamiltonian is constant over all time and the solution converges to equilibrium,  a single stationary particle. These perfect solutions are useful in that they display the most energetically efficient way of reaching equilibrium, i.e. they are the slowest to do so. We provide sufficient conditions for finite time collapse in the discrete case in Theorem \ref{discrete-sufficient}, whereas in the general case these conditions are more restrictive and spelled out in Theorem \ref{continuous-sufficient}. In the discrete case we also have a necessary condition for finite time collapse in Theorem \ref{discrete-necessary}, whereas Theorem \ref{quad-envelope} produces a sharp quadratic envelope which a solution must remain within in order to collapse. We demonstrate various (counter)examples that illustrate the unique behavior of the repulsive scheme with the sticky condition, analytically and with a computer simulation. 
    \subsection{Preliminaries}
    We recommend \cite{Villani} as a good source for the facts stated below. 
        Throughout this paper, $\mathcal{P}_p(E)$ denotes the set of Borel probability measures $\mu$ on a Borel set $E\subseteq\mathbb{R}$ with finite $p$-moment ($\int_E|y|^p\mu(dy)<\infty$). $\mathcal{L}(0,1)$ denotes the Lebesgue measure on $(0,1)$.
            \begin{definition}
                For $\mu_1,\mu_2\in\mathcal{P}_p(\mathbb{R})$, the \emph{$p$-Wasserstein distance} between the measures is given by 
                \begin{align*}
                    W_p^p(\mu_1,\mu_2)=\inf\Bigg\{\int_{\mathbb{R}^2}|y_1-y_2|^p\nu(dy_1,dy_2)\ | \ \nu\in\mathcal{P}(\mathbb{R}), (\pi)_\#^k\nu=\mu_k\Bigg\},
                \end{align*}
                where $(\pi)_\#^k$ is the projection onto the $k$th coordinate ($k\in\{1,2\}$). 
            \end{definition}
            If $S:\mathbb{R}\rightarrow\mathbb{R}$ is a Borel map and $\mu\in\mathcal{P}(\mathbb{R})$, then the {\it push-forward} of $\mu$ by $S$ is the Borel probability measure $\nu$ defined by $\nu(B):=\mu(S^{-1}(B))$ for all Borel sets $B$. It is denoted by $\nu:=S_\#\mu$ and it can also be alternatively defined as the unique Borel probability which satisfies
            $$\int\varphi(y)\nu(dy)=\int\varphi(S(x))\mu(dx)\mbox{ for all }\varphi\in C_c(\mathbb{R}).$$
            The (right-continuous) cumulative distribution function of $\mu$ is $M^{\mu}(y):=\mu((-\infty,y])$ and its generalized inverse $N^{\mu}$ is defined on $(0,1)$ as 
            $$N^{\mu}(m):=\inf\{y\in\mathbb{R}\ :\ M^{\mu}(y)>m\}\mbox{ for all }m\in(0,1).$$ 
            It is the unique right-continuous, nondecreasing function on $(0,1)$ such that $N^\mu_\#\mathcal{L}(0,1)=\mu$.
            It is known that 
            \begin{align}\nonumber
                W_p^p(\mu_1,\mu_2)=\int_0^1|N^{\mu_1}(x)-N^{\mu_2}(x)|^pdx\mbox{ for all }1\leq p<\infty
            \end{align}
            and if $p=1$ we also have $W_1(\mu_1,\mu_2)=\|M^{\mu_1}-M^{\mu_2}\|_{L^1(\mathbb{R})}$.

\section{Sticky Particles and Generalized Sticky Particles Solutions}
    \subsection{The Projection Formula}
        In \cite{BLPTW} it is proved that the following projection formula gives a sticky particles solution to the \emph{attractive} PEP equations. Let $N_t$ be the optimal map which pushes forward the Lebesgue measure on $(0,1)$ onto $\rho_t$. Then, 
        \begin{equation}\label{proj-formula}
            N_t=\text{proj}_\mathcal{K}\Big(N_0+tv_0(N_0)+\frac{t^2}{2}(0.5-\Id)\Big), 
        \end{equation}
        where $\mathcal{K}$ is the space of essentially nondecreasing functions in $L^2(0,1)$, is the desired solution. The projection onto the space $\mathcal{K}$ can be taken as follows: 
        \begin{align}
            \proj_\mathcal{K}(f)=\frac{d^+}{dm}F^{**}(m),
        \end{align}
        where $F^{**}$ is the convex envelope of the primitive of $f$ vanishing at zero.

        One might expect a similar formula to work for the repulsive case \eqref{PEP}, perhaps by replacing $(0.5-\Id)$ by $(\Id-0.5)$, but this is in fact false. This can be seen by considering a simple example in both the repulsive and attractive cases. 
        \begin{example}\label{attractive-repulsive-comparison}
            Let $\rho_0=\frac{1}{2}\delta_{-1}+\frac{1}{2}\delta_1$, $v_0=-\frac{\sqrt{2}}{2}\sgn$. The SPS for the attractive PEP equations with initial condition $(\rho_0,v_0)$ is 
            \begin{align*}
                \rho_t=\begin{cases}\frac{1}{2}\delta_{-1+\frac{\sqrt{2}}{2}t+\frac{t^2}{8}}+\frac{1}{2}\delta_{1-\frac{\sqrt{2}}{2}t-\frac{t^2}{8}} &\text{ for }0\leq t<4-2\sqrt{2}, 
                    \\
                    \delta_0, &\text{ for }t\geq4-2\sqrt{2},
                \end{cases}
            \end{align*}
            and that of the repulsive PEP is 
            \begin{align*}
                \rho_t=\begin{cases}\frac{1}{2}\delta_{-1+\frac{\sqrt{2}}{2}t-\frac{t^2}{8}}+\frac{1}{2}\delta_{1-\frac{\sqrt{2}}{2}t+\frac{t^2}{8}} &\text{ for }0\leq t<2\sqrt{2}, 
                    \\
                    \delta_0, &\text{ for }t\geq2\sqrt{2}.
                \end{cases}
            \end{align*}
        \end{example}
        These solutions look similar, but there is an important difference. Think of each $\delta$ as a particle, and consider its parabolic trajectory as though particles were to pass through each other, (ignoring the sticky condition) $g_1$ and $g_2$, as can be seen in figure \ref{fig:2-particle-ghost}. One can think of the projection formula as essentially working by taking this trajectory and averaging the mass that has crossed over when projecting onto $\mathcal{K}$. The difference, which can be seen in figure \ref{fig:2-particle-ghost}, is that in the repulsive case, there is no guarantee that $g_1(t)$ and $g_2(t)$ remain crossed after intersecting.\footnote{Replacing the space $\mathcal{K}$ with a different fixed space of functions $\mathcal{J}$ will not rectify this: observe in the repulsive solution in Example \ref{attractive-repulsive-comparison} that, for $G_t=N_0+tv_0(N_0)+\frac{t^2}{2}(0.5-\Id)$, $G_0=G_4$; however, we would need $\proj_\mathcal{J}G_0=N_0$ and $\proj_\mathcal{J}G_4=0$, (since $0_\#\mathcal{L}(0,1)=\delta_0$) which is not possible.} Upon uncrossing, the projection formula will yield two separate particles again, violating the sticky condition.

        \begin{figure}
            \centering
            \begin{tikzpicture}[scale=1]
                \draw (-1.75-3,0)--(1.75-3,0);
                \draw (0-3,-0.25)--(0-3,3);
                \draw (0.2-3,2.8) node {$t$};
                \draw (1.5-3+0.1,-0.3) node {$y$};

                \draw (-1.3-3,2) node {$g_2$};
                \draw (1.5-3,1.5) node {$g_1$};
                
                \draw plot[smooth, domain=-1:0] (\x-3, {2*sqrt(\x+2)-2});
                \draw plot[smooth, domain=0:1] (\x-3, {2*sqrt(-\x+2)-2});

                \draw[dashed] plot[smooth, domain=0:1.7] (\x-3, {2*sqrt(\x+2)-2});
                \draw[dashed] plot[smooth, domain=-1.7:0] (\x-3, {2*sqrt(-\x+2)-2});

                \draw[fill] (-1-3,0) circle [radius=2pt];
                \draw[fill] (1-3,0) circle [radius=2pt];
                \draw[fill] (0-3,0.8284) circle [radius=2.828pt];

                \draw(-4,-0.2)--(-4,0.2); 
                \draw(-2,-0.2)--(-2,0.2); 
                \draw (-4,-0.4) node {$\scriptstyle-1$};
                \draw (-2,-0.4) node {$\scriptstyle1$};

                \draw(-3-0.2,0.8284)--(-3+0.2,0.8284); 
                \draw (-3.75,0.8) node {$\scriptstyle4-2\sqrt{2}$};

                \*\*\*\*\*\*\*\*\*\*\*\*\*\*\*

                \draw (-1.75+3,0)--(1.75+3,0);
                \draw (0+3,-0.25)--(0+3,5);
                \draw (0.2+3,4.8) node {$t$};
                \draw (1.5+3+0.1,-0.3) node {$y$};

                \draw (-1.45+3,3.93) node {$g_1$};
                \draw (1.45+3,3.93) node {$g_2$};
            
                \draw plot[smooth, domain=-1:0] (\x+3, {-2*sqrt(-\x)+2});
                \draw plot[smooth, domain=0:1] (\x+3, {-2*sqrt(\x)+2});

                \draw[dashed] plot[smooth, domain=-1.7:0] (\x+3, {2*sqrt(-\x)+2});
                \draw[dashed] plot[smooth, domain=0:1.7] (\x+3, {2*sqrt(\x)+2});

                \draw[fill] (-1+3,0) circle [radius=2pt];
                \draw[fill] (1+3,0) circle [radius=2pt];
                \draw[fill] (0+3,2) circle [radius=2.828pt];

                \draw(2,-0.2)--(2,0.2); 
                \draw(4,-0.2)--(4,0.2); 
                
                \draw (2,-0.4) node {$\scriptstyle-1$};
                \draw (4,-0.4) node {$\scriptstyle1$};

                \draw(3-0.2,2)--(3+0.2,2); 
                \draw (3.55,2.05) node {$\scriptstyle2\sqrt{2}$};

            \end{tikzpicture}
            \caption{The trajectories in Example \ref{attractive-repulsive-comparison}. Left is the attractive case, right is the repulsive case.}
            \label{fig:2-particle-ghost}
        \end{figure}
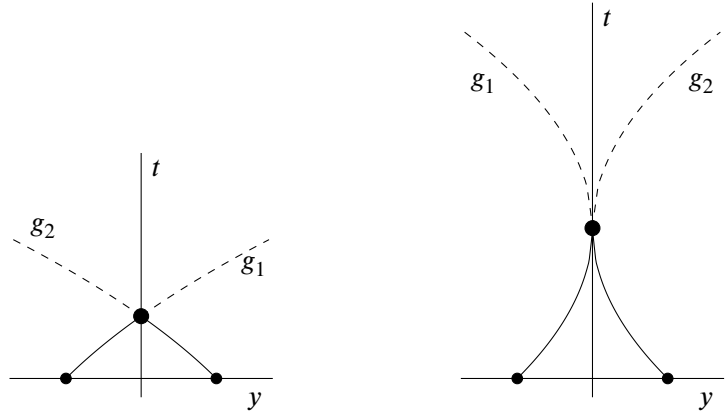

    \subsection{Describing Solutions}
        
        While the projection formula does not give a sticky particles solution, one can still find explicit solutions in certain cases. Taken together, the following properties of the repulsive PEP system are useful for this.
        \begin{proposition}\label{properties}
            Consider $(\rho_t,v_t)$, a distributional solution to the repulsive PEP equations with initial conditions $(\rho_0,v_0)$. Let $\rho_0\in\mathcal{P}_2(\mathbb{R})$, $v_0\in L^2(\rho_0)$. The following are properties of $(\rho_t,v_t)$. 
            \begin{enumerate}
                \item \textbf{Kinematic Law:}\label{particle-kinematics} Suppose $\rho_0=\sum_{i\in\mathbb{N}}m_i\delta_{y_i(0)}$. If there are no collisions in $(0,t_0)$, then the sticky particles solution in $(0,t_0)$ is given by $\rho_t=\sum_{i\in\mathbb{N}}m_i\delta_{y_i(t)}$, where
                \begin{align} \label{kinematics}
                    y_i(t)=y_i(0)+tv_0\big(y_i(0)\big)+\frac{t^2}{4}\Big(\sum_{\{j | y_j(0)<y_i(0)\}}m_j\ -\sum_{\{j | y_j(0)>y_i(0)\}}m_j\Big). 
                \end{align}       
                That is, the above solution holds up until $y_i(t)=y_j(t)$, $i\neq j$. $v_t$ is defined $\rho_t$ almost everywhere by $v(y_i(t))=y_i'(t)$.
                \item \textbf{Galilean Invariance:}\label{galilean-invariance} Define the initial center of mass and total momentum for the initial condition $(\rho_0,v_0)$ by 
                    \begin{align}
                        Y_0=\int y\rho_0(dy), \quad m_0=\int v_0(y)\rho_0(dy).
                    \end{align}
                    Then, $(\bar{\rho},\bar{v})$, where 
                    \begin{align*}
                        \bar{\rho}_t&=\big(\text{Id}-Y_0-tm_0\big)_\#\rho_t, \quad \bar{v}_t=v_t(\cdot+Y_0+tm_0)-m_0,
                    \end{align*}
                    is a solution to the repulsive PEP equations for the initial condition $(\bar{\rho}_0,\bar{v}_0)=((\text{Id}-Y_0)_\#\rho_0,v_0(\cdot+Y_0)-m_0)$ if and only if $(\rho,v)$ is a solution for the initial condition $(\rho_0,v_0)$. Further, for all $t\geq0$,\footnote{This implies that the center of mass for a general solution $(\rho_t,v_t)$ is $\bar{y}(t)=Y_0+tm_0$.}
                    \begin{align*}
                        \int y \bar{\rho}_t(dy)=\int \bar{v}_t(y) \bar{\rho}_t(dy)=0.
                    \end{align*}
                \item \textbf{Equilibrium Solutions:}\label{only-equilibrium} The only equilibrium solutions are point masses. Specifically, if we require that the center of mass and total momentum are $0$ as in \ref{galilean-invariance}, the only equilibrium is $(\delta_0,0)$.
            \end{enumerate}
        \end{proposition}
        \begin{proof} We state here the weak formulation of \eqref{PEP}: the continuity equation is 
        \begin{align}\label{continuity-equation}
            \frac{d}{dt}\int\varphi(y)\rho_t(dy)=\int\varphi'(y)v_t(y)\rho_t(dy), 
        \end{align}
        for all $\varphi\in C_b^1(\mathbb{R})$, and the momentum equation is given by, for all $\varphi\in C_b^1(\mathbb{R})$,
        \begin{align}\label{momentumn-equation}
            \frac{d}{dt}\int\varphi(y)v_t(y)\rho_t(dy)=\int\varphi'(y)v_t^2(y)\rho_t(dy)+\frac{1}{2}\int\varphi(y)\big(\text{sgn}*\rho_t\big)(y)\rho_t(dy). 
        \end{align}
        The time derivatives are to be understood as derivatives of absolutely continuous functions 
        on $\mathbb{R}$. 
        \begin{enumerate}
            \item  For $t\in(0,t_0)$, $\varphi$ bounded,
                \begin{align*}
                    \frac{d}{dt}\int\varphi(y)\rho_t(dy)=\ &\frac{d}{dt}\sum_{i}m_i\varphi\big(y_i(t)\big)=\sum_{i}m_i\varphi'\big(y_i(t)\big)y_i'(t)
                    \\
                    =\ &\sum_{i}m_i\varphi'\big(y_i(t)\big)v_t\big(y_i(t)\big)=\int\varphi'(y)v_t(y)\rho_t(dy),
                \end{align*}
                and, since $v_t\in L^2(\rho_t)$,  
                \begin{align*}
                    \frac{d}{dt}\int\varphi(y)v_t(y)&\rho_t(dy)=\ \frac{d}{dt}\sum_{i}m_i\varphi\big(y_i(t)\big)v_t\big(y_i(t)\big)
                    \\
                    =&\ \sum_{i}m_i\varphi'\big(y_i(t)\big)y_i'(t)y_i'(t)+\sum_{i}m_i\varphi\big(y_i(t)\big)y_i''(t)
                    \\
                    =&\ \sum_{i}m_i\varphi'\big(y_i(t)\big)y_i'(t)^2+\frac{1}{2}\sum_{i}m_i\varphi\big(y_i(t)\big)\Big(\sum_{y_j(0)<y_i(0)}m_j\ -\sum_{y_j(0)>y_i(0)}m_j\Big)
                    \\
                    =&\ \int\varphi'(y)v_t^2(y)+\frac{1}{2}\int\varphi(y)\big(\text{sgn}*\rho_t\big)\rho_t(dy). 
                \end{align*}
                It is clear that $\rho_t\in\mathcal{P}_2(\mathbb{R})$. As for the sticky condition, we have $N_t(x)=\sum_{i}y_i(t)\chi_I(x)$ where $I={\big(\sum_{y_j(0)<y_i(0)}m_j,\sum_{y_j(0)\leq y_i(0)}m_j\big)}$, so the condition is satisfied up until $y_i(t_0)=y_j(t_0)$, $i\neq j$. 
           
            \item Let $\varphi\in C_c^1([0,T)\times\mathbb{R})$ and set $\tilde{\varphi}(t,y)=\varphi(y-Y_0-tm_0)$. The continuity equation for $(\rho,v)$ yields 
            \begin{align*}
            \int_0^T \!\! \int\big[\partial_t\tilde{\varphi}+v\partial_y\tilde{\varphi}\big]d\rho_tdt=-\int\tilde{\varphi}(0,y)\rho_0(dy),
            \end{align*}
            which is equivalent to 
            \begin{align}\nonumber
                \int_0^T \!\! \int\big\{\partial_t\varphi(t,S_t(y))+\big[v(t,y)-m_0\big]\partial_y\varphi(t,S_t(y))\big\}\rho_t(dy)dt=-\int\varphi(0,S_0(y))\rho_0(dy),
            \end{align}
            where $S_t(y)=y-Y_0-tm_0$, with $S_t^{-1}(z)=z+Y_0+tm_0$, so 
            \begin{equation}\label{new-cont-eq}
                \int_0^T \!\! \int\big\{\partial_t\varphi(z,t)+\big[v(t,S_t^{-1}(z))-m_0\big]\partial_z\varphi(t,z)\big\}\bar{\rho}_t(dz)dt=-\int\varphi(0,z)\bar{\rho}_0(dz);
            \end{equation}
            that is, $(\bar{\rho},\bar{v})$ satisfies the continuity equation. As for the momentum equation, we have 
            \begin{align*}
                A+\int_0^T \!\! \int\big[v\partial_t\tilde{\varphi}+v^2\partial_y\tilde{\varphi}\big]\rho_t(dy)dt=-\int v_0(y)\tilde{\varphi}(0,y)\rho_0(dy),
            \end{align*}
            which amounts to
            \begin{align*}
                A+\int_0^T& \!\! \int\big\{\partial_t\varphi(t,z)v(t,S_t^{-1}(z))+\partial_z\varphi (t,z)\big[v^2(t,S_t^{-1}(z))-m_0v(t,S_t^{-1}(z))\big]\big\}\bar\rho_t(dz)dt
                \\&
                =-\int v_0(S_0^{-1}(z))\varphi(0,z)\bar \rho_0(dz),
            \end{align*}
            where $A$ is the interaction term
            \begin{align*}A&=-\int_0^T \!\! \iint \! \tilde \varphi(t,y)\mathrm{sgn} (y-z)\rho_t(dz)\rho_t(dy)dt
            =-\int_0^T \!\! \iint \! \varphi(t,S_t(y))\mathrm{sgn}(S_t(y)-S_t(z))\rho_t(dz)\rho_t(dy)dt
            \\&
            =-\int_0^T \!\! \iint \! \varphi(t,y)\mathrm{sgn} (y-z)\bar\rho_t(dz)\bar\rho_t(dy)dt.
            \end{align*}
            Now multiply \eqref{new-cont-eq} by $m_0$ and subtract from the equation above to get the momentum equation for $(\bar\rho,\bar v)$. Since $S_t$ is invertible for all $t\geq0$, it is obvious that the converse holds as well.
            \item See the proof of Proposition 3.1 in \cite{BLPTW}.\qedhere  
        \end{enumerate}
        \end{proof}
        We will often make use of the above proposition without explicit reference; it allows for us to employ physical reasoning about particles in the case of discrete solutions. When doing so, we may say that particles $i$ and $j$ \emph{collide at time $t_0$} to mean $y_i(t_0)=y_j(t_0)$ and $y_i(t)\neq y_j(t)$ for all $0\leq t<t_0$. We call $y_i$ the \emph{trajectory} of the $i$th particle. 

        There is one other relevant property which is obvious: that the evolution of a solution at $(\rho_{t_0}, v_{t_0})$ for $t \ge t_0$ will be identical (but for shifted $t$) to the entire evolution of the solution with initial conditions $(\tilde{\rho}_0=\rho_{t_0}, \tilde{v}_0=v_{t_0})$. This property is established under the name of the \textbf{semigroup property} in \cite{BLPTW}. This is used implicitly when we apply kinematics laws to any collisionless time interval (not just those starting at 0), and will be used explicitly in Lemma \ref{PC1}, Lemma \ref{PC}, and Theorem \ref{perfect_e_and_u}.
    \subsection{Ill-Posedness and Discrete Approximations}
        The repulsive problem is not well posed for some initial conditions. Consider the following:
        \begin{example}\label{ill-posed}
            Let $\varepsilon>0$. The SPS to the repulsive PEP problem with initial condition $(\rho_0^\varepsilon=\frac{1}{2}\delta_{-1-\varepsilon}+\frac{1}{2}\delta_{1+\varepsilon},v_0=-\frac{\sqrt{2}}{2}\sgn)$ is given by 
            \begin{align*}
                \rho_t^\varepsilon=\frac{1}{2}\delta_{-1-\varepsilon+\frac{\sqrt{2}}{2}t-\frac{t^2}{8}}+\frac{1}{2}\delta_{1+\varepsilon-\frac{\sqrt{2}}{2}t+\frac{t^2}{8}}
            \end{align*}
            for all $t\geq0$. 
        \end{example}
        Taking $\varepsilon$ to $0$ in Example \ref{ill-posed}, the initial conditions approach that of Example \ref{attractive-repulsive-comparison}. While the solution depends continuously on the initial conditions for $0\leq t\leq2\sqrt{2}$, we have, for all $t>2\sqrt{2}$, 
        \begin{align*}
            W_p(\rho_t^\varepsilon,\rho_t)=W_p(\rho_t^\varepsilon,\delta_0)=\Bigg(\int|y|^p\rho_t^\varepsilon(dy)\Bigg)^{1/p}=1+\varepsilon-\frac{\sqrt{2}}{2}t+\frac{t^2}{8}. 
        \end{align*}
        That is, for $\varepsilon>0$, $W_p(\rho_t^\varepsilon,\rho_t)$ increases quadratically in time after $t=2\sqrt{2}$. 

        In a similar manner, given a continuous initial distribution, one can construct a sequence of discrete approximating solutions, all of which satisfy the sticky condition, but which converge to a solution to the repulsive PEP equations which \emph{does not} satisfy the sticky condition; see Example \ref{bad-approximation}. 
        \begin{example}\label{bad-approximation}Consider the (continuous) initial conditions $\rho_0=\frac{1}{2}\chi_{[-1,1]}$, $v_0=\Id$. We can approximate this distribution by $\rho_0^{n}=\frac{1}{2n+1}\sum_{i=-n}^n\delta_{(i+\sgn(i))/n}$, $v_0^n(y)=\frac{\sgn(y)}{n}-y$. These discrete approximations give rise to discrete solutions $(\rho_t^n,v_t^n)$ which approximate the distribution
            \begin{align*}
                \rho_t=\begin{cases}\frac{2}{(t-2)^2}\chi_{[-\frac{1}{4}(t-2)^2,\frac{1}{4}(t-2)^2]}, &0\leq t<2,
                \\
                \frac{1}{2}\delta_{-\frac{1}{8}(t-2)^2}+\frac{1}{2}\delta_{\frac{1}{8}(t-2)^2}, &t\geq2.\end{cases}
            \end{align*}
            Note that the conditions in Proposition 3.3 of \cite{Tudorascu-2008} are satisfied by $\rho_t^n$. 
        \end{example}
        \begin{figure}
            \centering
            \begin{tikzpicture}[scale=1]
                \draw[draw=white,fill=black] plot[smooth,samples=100,domain=-1+3:1+3] (\x,{0}) -- plot[smooth,samples=100,domain=1+3:-1+3] (\x,{2-2*sqrt(abs(\x-3))});

                \draw (-2+3,0)--(2+3,0);
                \draw (0+3,-0.25)--(0+3,5);
                \draw (0.2+3,4.8) node {$t$};
                \draw (1.8+3,0.3) node {$y$};
            
                \draw plot[smooth, domain=0+3:0.9+3] (\x, {2.8284271*sqrt(\x-3)+2});
                \draw plot[smooth, domain=-0.9+3:0+3] (\x, {2.8284271*sqrt(-(\x-3))+2});
        
                \draw[fill] (-0.9+3,4.68328) circle [radius=2pt];
                \draw[fill] (0.9+3,4.68328) circle [radius=2pt];
                \draw[fill] (0+3,2) circle [radius=2.828pt];

                \draw(-2,-0.2)--(-2,0.2); 
                \draw(-4,-0.2)--(-4,0.2); 
                \draw(-3-0.08,2)--(-3+0.2,2); 

                \draw (-2,-0.4) node {$\scriptstyle-1$};
                \draw (-4,-0.4) node {$\scriptstyle1$};
                \draw (-3.15,2) node {$\scriptstyle2$};


                \draw (-2-3,0)--(2-3,0);
                \draw (0-3,-0.25)--(0-3,5);
                \draw (0.2-3,4.8) node {$t$};
                \draw (1.8-3,0.3) node {$y$};
            
                \draw plot[smooth, domain=0.2223-3:0.9-3] (\x, {2.3333+sqrt(7*(\x+3)-1.55555)});
                \draw plot[smooth, domain=0.2223-3:1-3] (\x, {2.3333-sqrt(7*(\x+3)-1.55555)});
                \draw (0.22222-3,2.308) -- (0.22222-3,2.36);
                \draw plot[smooth, domain=0.33364-3:1.333333-3] (\x, {2.33333-sqrt(4.6666*(\x+3)-0.777777)});
                \draw plot[smooth, domain=0.33364-3:0.66666-3] (\x, {2.3333-sqrt(14*(\x+3)-3.8888)});
                
                \draw[fill] (0.666666-3,0) circle [radius=1.5pt];
                \draw[fill] (1-3,0) circle [radius=1.5pt];
                \draw[fill] (1.333333-3,0) circle [radius=1.5pt];
                \draw[fill] (0.33364-3,1.4502) circle [radius=2.5pt];

                \draw plot[smooth, domain=-0.9-3:-0.2223-3] (\x, {2.3333+sqrt(-7*(\x+3)-1.55555)});
                \draw plot[smooth, domain=-1-3:-0.2223-3] (\x, {2.3333-sqrt(-7*(\x+3)-1.55555)});
                \draw (-0.22222-3,2.25) -- (-0.22222-3,2.39);
                \draw plot[smooth, domain=-1.333333-3:-0.33364-3] (\x, {2.33333-sqrt(-4.6666*(\x+3)-0.777777)});
                \draw plot[smooth, domain=-0.66666-3:-0.33364-3] (\x, {2.3333-sqrt(-14*(\x+3)-3.8888)});
                
                \draw[fill] (-0.666666-3,0) circle [radius=1.5pt];
                \draw[fill] (-1-3,0) circle [radius=1.5pt];
                \draw[fill] (-1.333333-3,0) circle [radius=1.5pt];
                \draw[fill] (-0.33364-3,1.4502) circle [radius=2.5pt];

                \draw[fill] (0-3,0) circle [radius=1.5pt];

                \draw(2,-0.2)--(2,0.2); 
                \draw(4,-0.2)--(4,0.2); 
                \draw(3-0.2,2)--(3+0.2,2); 

                \draw (2,-0.4) node {$\scriptstyle-1$};
                \draw (4,-0.4) node {$\scriptstyle1$};
                \draw (3.3,2) node {$\scriptstyle2$};
            \end{tikzpicture}

            \caption{Left: The discrete approximation $\rho_t^3$ from Example \ref{bad-approximation}. Right: $\rho_t$ as in Example \ref{bad-approximation}.}
            \label{fig:continuous-split}
        \end{figure}
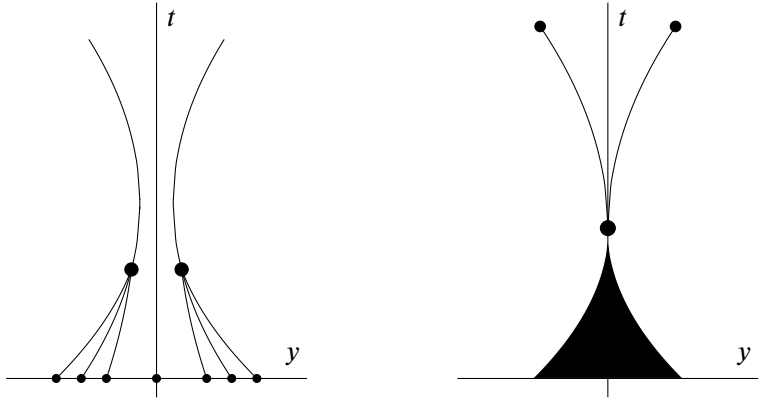

        One may wish to work not with the class of sticky particles solutions, $\mathcal{S}$, but its closure, $\bar{\mathcal{S}}$: solutions which can be obtained by approximations with sticky particles solutions. However, such solutions are not unique, even in the discrete case. 
        \begin{example}
            Consider the initial condition $(\delta_0,0)$. We claim that, while there is only one solution in $\mathcal{S}$ satisfying this condition, there are infinitely many in $\bar{\mathcal{S}}$. Indeed, let $c>0$ and consider the sticky particles solution to the initial conditions 
            \begin{align*}
                \rho_0^n&=\frac{1}{2}\delta_{-\frac{1}{n}}+\frac{1}{2}\delta_\frac{1}{n}, \quad v_0=c\cdot\sgn. 
            \end{align*}
            Taking $n\to\infty$, we see that $W(\rho_0^n,\delta_0)<1/n\to0$. Each choice of $c$ gives a different solution to the repulsive PEP equations, none of which are SPS. 
        \end{example}
        Therefore, we must take care to specify exactly what class of solutions our results apply to:
        \begin{definition}
            The following are the various classes of solutions $(\rho,v)$ which we will work with. 
            \begin{enumerate}
                \item \textbf{Discrete solutions} will refer to solutions for which $\rho_t$ is a linear combination of Dirac measures. By Proposition \ref{properties}, \ref{particle-kinematics}, if $\rho_0$ is discrete, so is $\rho_t$. \textbf{Finite discrete solutions} have only finitely many nonzero masses.  We denote the class of finite discrete solutions by $\mathcal{S}_d$.
                \item $\bar{\mathcal{S}}$ denotes the class of \textbf{Generalized Sticky Particles Solutions}, (GSPS), i.e. solutions $(\rho,v)$ in the sense of distributions
                    such that there exists $\{(\rho^n,v^n)\}_n\subset\mathcal{S}_d$ which converges strongly to $(\rho,v)$. We define this to mean 
                    $$\max_{t\in[0,T]}W_2(\rho_t^n,\rho_t)\rightarrow 0\mbox{ and }v_t^n\rho_t^n\mbox{ converges strongly to }v_t\rho_t\mbox{ for a.e. }t\in[0,T].$$ 
                    This strong convergence is to be understood in the following sense: $v_t^n\rho_t^n$ converges to $v_t\rho_t$ in the sense of distributions and $\|v_t^n\|_{L^2(\rho^n_t)}$ converges to $\|v_t\|_{L^2(\rho_t)}$  for a.e. $t\in[0,T].$
                \item $\mathcal{S}$ will refer to the class of \textbf{Sticky Particles Solutions} (SPS), any solutions in $\bar{\mathcal{S}}$ which satisfy the \textbf{sticky condition}. In the continuous case the sticky condition can be formulated in terms of the optimal maps $N_t$ that push forward the uniform measure on $(0,1)$ to $\rho_t$. It reads
                    $$N_s(a)=N_s(b)\mbox{ for some }s\geq0,\ 0<a<b<1\Rightarrow N_t(a)=N_t(b)\mbox{ for all }t\geq s.$$
            \end{enumerate}
        \end{definition}
        Existence and uniqueness of SPS for the presureless Euler and attractive pressureless Euler-Poisson systems are now well understood and come up as an immediate consequence of the projection formula \eqref{proj-formula} (the quadratic term in $t$ is missing for the PE system) \cite{Natile-Savare}, \cite{Suder-Tudorascu}, \cite{BLPTW}.
        As discussed earlier in this section, such  formula is not available in the repulsive case; for all our proofs we must rely on discrete approximations by solutions in $\mathcal{S}_d$.
        We also note here that in \cite{Suder-Tudorascu} the authors proved that the SPS to PE is also unique among all distributional solutions that satisfy a strong initial continuity of energy condition, plus the so-called Oleinik condition (see \cite{Suder-Tudorascu}). However, it is a simple exercise to show that the SPS to the repulsive equations fail to satisfy the Oleinik condition.
    \subsection{Instantaneous Failure of the Projection Formula}
        Before moving on, we will present an example in which \eqref{proj-formula} breaks instantaneously. Observe that the projection formula does not hold in the repulsive case for $t>t_c=2\sqrt{2}$ in Example \ref{attractive-repulsive-comparison}. The goal here is to create a DSPS with nested situations like this, having $t_c$ arbitrarily small. Indeed, put a particle of mass $1/4$ at $y=1$, mass $1/8$ at $y=1/3$, mass $1/16$ at $y=1/9$, and so on. Mirroring this on the negatives, we have constructed $\rho_0\in\mathcal{P}_2(\mathbb{R})$. The initial acceleration of each particle may be calculated by a geometric series. Then, give each particle just enough velocity to reach the origin, i.e., so that its parabolic trajectory has its vertex at $y=0$.\footnote{See Definition \ref{glancing-perfect} and section \ref{perfect-section} for more on solutions like this.} In all, we obtain the following initial conditions and solution: 
        \begin{example}\label{breaks-instantaneous}
            The DSPS to \eqref{PEP} with initial conditions given by
            \begin{align*}
                \rho_0=\sum_{i=1}^\infty\frac{\delta_{3^{1-i}}+\delta_{-(3^{1-i})}}{2^{i+1}}, \quad\quad v_0(\pm3^{1-i})=\mp\frac{3\sqrt{2}}{2}\big(\sqrt{6}\big)^{-i}
            \end{align*}
            has the distribution (shown in figure \ref{projection-break-figure}) given by 
            \begin{align*}
                \rho_t=\begin{cases}
                    \displaystyle 2^{-i^*(t)}\delta_0\ +\ \sum_{i=1}^{i^*(t)}\frac{\delta_{y_i(t)}+\delta_{-y_i(t)}}{2^{i+1}}, & 0<t<\frac{4\sqrt{3}}{3},
                    \\
                    \delta_0, & t\geq\frac{4\sqrt{3}}{3},
                \end{cases} 
            \end{align*}
            where $\ \ \quad \displaystyle i^*(t)=\bigg\lceil\log_{\frac{\sqrt2}{\sqrt3}}\big(\frac{\sqrt2}{4}t\big)\bigg\rceil-1 \quad$ \ \ and \ \  $\quad\displaystyle y_i(t)=3^{1-i}\ -\ \frac{3\sqrt{2}}{2}\big(\sqrt{6}\big)^{-i}t\ +\ \frac{3}{8}2^{-i}t^2$.
        \end{example}
        Observe that the mirrored particles indexed $i$ reach $y=0$ at time $t_c(i)=2\sqrt{2}\Big(\frac{\sqrt2}{\sqrt3}\Big)^i$. Because $t_c(i)\to0$ as $i\to\infty$, for any $t>0$, the phenomenon we observed in Example \ref{attractive-repulsive-comparison} has already occurred. 

        \begin{figure}
            \centering
            \begin{tikzpicture}[scale=2]
                \draw (-1.1*3,0)--(1.1*3,0);
                \draw (0,-0.25)--(0,2.7);
                \draw (0.03*3,2.6) node {$t$};
                \draw (1.1*3,0.1) node {$y$};
                \draw (-1*3,0.1)--(-1*3,-0.1);
                \draw (-1*3,-0.2) node {$\scriptstyle-1$};
                \draw (1*3,0.1)--(1*3,-0.1);
                \draw (1.*3,-0.2) node {$\scriptstyle1$};

                \draw (-1,0.1)--(-1,-0.1);
                \draw (-1,-0.25) node {$\scriptstyle-\frac{1}{3}$};
                \draw (1,0.1)--(1,-0.1);
                \draw (1.,-0.25) node {$\scriptstyle\frac{1}{3}$};

                \draw[fill] (3*1,0) circle [radius=1pt];
                \draw[fill] (3*1/3,0) circle [radius=0.70710678118pt];
                \draw[fill] (3*1/9,0) circle [radius=0.5pt];
                \draw[fill] (3*1/27,0) circle [radius=0.35355339059pt];
                \draw[fill] (3*1/81,0) circle [radius=0.25pt];
                \draw[fill] (3*1/243,0) circle [radius=0.17677669529pt];
                \draw[fill] (3*1/729,0) circle [radius=0.125pt];
                \draw[fill] (3*1/2187,0) circle [radius=0.08838834764pt];
                \draw[fill] (3*1/6561,0) circle [radius=0.0625pt];

                \draw plot[smooth, variable =\y, domain=0:2.82842*(0.81649^1)] ({3*(sqrt(3)^(1-1)-(\y*sqrt(3)*(sqrt(2)^(-1))/(2*sqrt(2))))^2},{\y});
                \draw plot[smooth, variable =\y, domain=0:2.82842*(0.81649^2)] ({3*(sqrt(3)^(1-2)-(\y*sqrt(3)*(sqrt(2)^(-2))/(2*sqrt(2))))^2},{\y});
                \draw plot[smooth, variable =\y, domain=0:2.82842*(0.81649^3)] ({3*(sqrt(3)^(1-3)-(\y*sqrt(3)*(sqrt(2)^(-3))/(2*sqrt(2))))^2},{\y});
                \draw plot[smooth, variable =\y, domain=0:2.82842*(0.81649^4)] ({3*(sqrt(3)^(1-4)-(\y*sqrt(3)*(sqrt(2)^(-4))/(2*sqrt(2))))^2},{\y});
                \draw plot[smooth, variable =\y, domain=0:2.82842*(0.81649^5)] ({3*(sqrt(3)^(1-5)-(\y*sqrt(3)*(sqrt(2)^(-5))/(2*sqrt(2))))^2},{\y});
                \draw plot[smooth, variable =\y, domain=0:2.82842*(0.81649^6)] ({3*(sqrt(3)^(1-6)-(\y*sqrt(3)*(sqrt(2)^(-6))/(2*sqrt(2))))^2},{\y});
                \draw plot[smooth, variable =\y, domain=0:2.82842*(0.81649^7)] ({3*(sqrt(3)^(1-7)-(\y*sqrt(3)*(sqrt(2)^(-7))/(2*sqrt(2))))^2},{\y});

                
                \draw[fill] (-3*1,0) circle [radius=1pt];
                \draw[fill] (-3*1/3,0) circle [radius=0.70710678118pt];
                \draw[fill] (-3*1/9,0) circle [radius=0.5pt];
                \draw[fill] (-3*1/27,0) circle [radius=0.35355339059pt];
                \draw[fill] (-3*1/81,0) circle [radius=0.25pt];
                \draw[fill] (-3*1/243,0) circle [radius=0.17677669529pt];
                \draw[fill] (-3*1/729,0) circle [radius=0.125pt];
                \draw[fill] (-3*1/2187,0) circle [radius=0.08838834764pt];
                \draw[fill] (-3*1/6561,0) circle [radius=0.0625pt];

                \draw plot[smooth, variable =\y, domain=0:2.82842*(0.81649^1)] ({-3*(sqrt(3)^(1-1)-(\y*sqrt(3)*(sqrt(2)^(-1))/(2*sqrt(2))))^2},{\y});
                \draw plot[smooth, variable =\y, domain=0:2.82842*(0.81649^2)] ({-3*(sqrt(3)^(1-2)-(\y*sqrt(3)*(sqrt(2)^(-2))/(2*sqrt(2))))^2},{\y});
                \draw plot[smooth, variable =\y, domain=0:2.82842*(0.81649^3)] ({-3*(sqrt(3)^(1-3)-(\y*sqrt(3)*(sqrt(2)^(-3))/(2*sqrt(2))))^2},{\y});
                \draw plot[smooth, variable =\y, domain=0:2.82842*(0.81649^4)] ({-3*(sqrt(3)^(1-4)-(\y*sqrt(3)*(sqrt(2)^(-4))/(2*sqrt(2))))^2},{\y});
                \draw plot[smooth, variable =\y, domain=0:2.82842*(0.81649^5)] ({-3*(sqrt(3)^(1-5)-(\y*sqrt(3)*(sqrt(2)^(-5))/(2*sqrt(2))))^2},{\y});
                \draw plot[smooth, variable =\y, domain=0:2.82842*(0.81649^6)] ({-3*(sqrt(3)^(1-6)-(\y*sqrt(3)*(sqrt(2)^(-6))/(2*sqrt(2))))^2},{\y});
                \draw plot[smooth, variable =\y, domain=0:2.82842*(0.81649^7)] ({-3*(sqrt(3)^(1-7)-(\y*sqrt(3)*(sqrt(2)^(-7))/(2*sqrt(2))))^2},{\y});
                

                \draw[fill] (0,2.30940107676) circle [radius=2pt];
                \draw[fill] (0,1.88561808316) circle [radius=1.41421356237pt];
                \draw[fill] (0,1.53960071784) circle [radius=1pt];
                \draw[fill] (0,1.25707872211) circle [radius=0.70710678118pt];
                \draw[fill] (0,1.02640047856) circle [radius=0.5pt];
                \draw[fill] (0,0.838052481406) circle [radius=0.35355339059pt];
                \draw[fill] (0,0.684266985706) circle [radius=0.25pt];
                \draw[fill] (0,0.558701654271) circle [radius=0.17677669529pt];
                \draw[fill] (0,0.456177990471) circle [radius=0.125pt];
                \draw[fill] (0,0.372467769514) circle [radius=0.08838834764pt];
                \draw[fill] (0,0.304118660314) circle [radius=0.0625pt];
                
                \draw[white,fill=white] (0.13,1.885) circle (0.05);
                
                \draw[white,fill=white] (0.3,1.55) circle (0.065);

                \draw(-0.1,2.309)--(0.1,2.309);
                \draw (0.6,2.309) node {$t_c(1)=\frac{4\sqrt{3}}{3}$};

                \draw(-0.1,1.885)--(0.1,1.885);
                \draw (0.25,1.885) node {$\scriptstyle t_c(2)$};

                \draw(-0.1,1.539)--(0.1,1.539);
                \draw (0.25,1.539) node {$\scriptscriptstyle t_c(3)$};

                \draw (0.25,1.25707872211) node {$\scriptstyle\vdots$};

            \end{tikzpicture}

            \caption{The distribution described in Example \ref{breaks-instantaneous}.}
            \label{projection-break-figure}
        \end{figure}
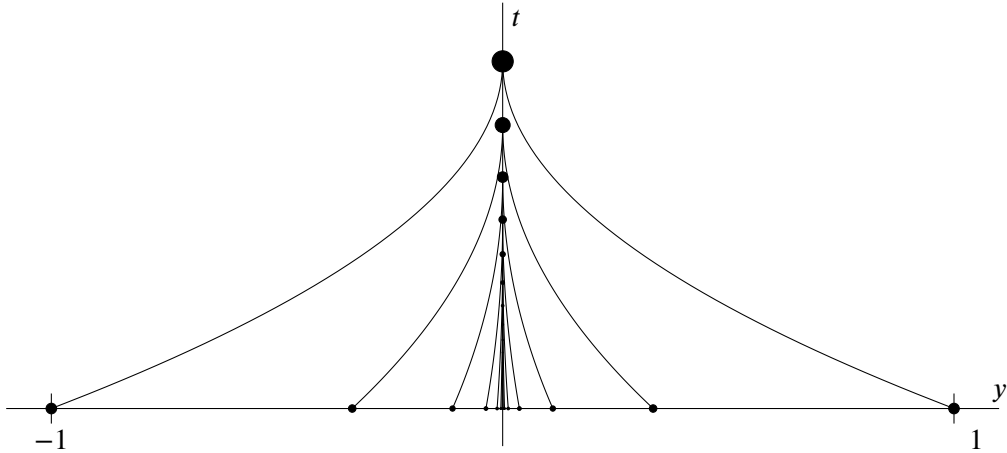   
        
        In the next section, we discuss a useful technique for examining the behavior of discrete solutions. 
    \subsection{Combining Particles in finite DSPS}\label{particle-combining-section}
    The following section will be primarily in service of establishing Lemma \ref{PC}, which will prove a useful tool in several future theorems. It shows that any contained group of particles can be replaced by a single particle at the center of mass of the group with the group's total mass and total momentum, with no change to the relevant behavior of the solution. 
    
    Throughout this section, we will consider a finite DSPS with the distribution $\rho_t=\sum_i m_i y_i(t)$, with $y_1(0)<y_2(0)<\ldots<y_n(0)$, often referring to the \emph{particles} $p_i=(m_i,y_i)$. We will use the convention that $y_i$ tracks the position of a particle after collisions, that is, it becomes equal to the position of the particle formed. We can therefore interpret the sticky condition as follows: 
    $$y_i(s)=y_j(s)\mbox{ for some }s\geq0,\ i\neq j\Rightarrow y_i(t)=y_j(t)\mbox{ for all }t\geq s.$$ Between collisions, the trajectory of a particle can be calculated using \eqref{kinematics}. Thus, the trajectory of a particle is piecewise quadratic. Finally, we will refer to the \textbf{ghost trajectory} of a particle, 
    \begin{align}\label{ghost-trajectory}
        g_i(t)=y_i(0)\ +\ ty_i'(0)\ +\ \frac{t^2}{4}\Big(\sum_{j<i}m_j-\sum_{j>i}m_j\Big).
    \end{align}
    This is the initial parabolic trajectory extrapolated over all time (i.e., the dotted lines in figure \ref{fig:2-particle-ghost}).

        \begin{lemma}\label{PC1}
        
            Consider an initial adjacent $l$ particle subset $P = \{p_q,\ldots,p_{q+l-1}\}$ and let $t_P$ be the maximal time such that for $t<t_P$, particles in $P$ only have collisions with other particles in $P$. Then the trajectory of the center of mass of $P$ until $t_P$, denoted $y_P(t)$, is equal to the trajectory of the center of mass of the ghost trajectories:
        \begin{align*}\label{Y-P}
            y_P(t) &= \frac{\sum_{p_i\in P}m_iy_i(t)}{\sum_{p_i\in P}m_i} = \frac{\sum_{p_i\in P}m_ig_i(t)}{\sum_{p_i\in P}m_i}
            \end{align*}
        \end{lemma}
    
        \begin{proof}
            We will prove this by an examination of the first collision in $\rho$, allowing us to induct on the number of particles through an application of the semigroup property.
        
            For this and future proofs in this section, we will re-index for ease of notation: let $i=1-(q-1)$, $2-(q-1),\ldots,$ $n-(q-1)$, such that $p_1$ is the first particle of $P$. We proceed by induction on $l$, for which the $1$ particle base case is trivial. Now, assume this holds true for $l=k$, and we will show it for $l=k+1$. If there are no collisions within $P$ for $t < t_P$ the result is trivial, so we assume there is a collision within $P$ before $t_P$. Let $t_c$ be the time of first collision within $P$. We have the result trivially for $t \le t_c$, so it remains to be shown for $t_c < t \le t_P$. Let $p_j$ and $p_{j+1}$ be a pair of particles involved in the first collision in $P$ (If there are multiple collisions at $t_c$, we can choose any adjacent pair which collides).\footnote{One may complain that singling out two particles from the first collision in $P$ may not be accurate as, once the semigroup property is applied, the only composite particle which exists with its own trajectory is that of all particles involved in the first collision. We can address this with strong induction, reducing the system according to the first collision size. However, this argument has the exact same steps as the given, with far denser notation. For this reason, the two particle assumption is allowed.} After collision, these two particles combine into a new particle $\bar p$ with mass $\bar m = m_j + m_{j+1}$. We can calculate its ghost trajectory for $t \ge t_c$, setting our initial condition to $t_c$ (using the semigroup property) and recalling that collisions observe conservation of momentum. Now considering the solution originating at $\tilde{\rho}_0=\rho_{t_c}$, we have $y_{\bar{p}}(0) = y_j(0)=y_{j+1}(0)$, $y'_{\bar{p}}(0) = \frac{m_jy_j'(t) + m_{j+1}y'_{j+1}(t)}{m_j + m_{j+1}}$, and 
            \begin{align*}
                \sum_{i<j}m_i-\sum_{i>j+1}m_i                =\frac{m_j\left(\sum_{i < j}m_i - m_{j+1}-\sum_{i > {j+1}}m_i\right)+m_{j+1}\left(\sum_{i < j}m_i + m_{j}-\sum_{i > {j+1}}m_i\right)}{m_j+m_{j+1}}.
            \end{align*}
            Thus,
            \begin{align*}
                g_{\bar{p}}(t) 
                &= y_{\bar{p}}(0)\ +\ ty'_{\bar{p}}(0)\ +\ \frac{t^2}{4} \left(\sum_{i<j}m_i-\sum_{i>j+1}m_i \right)  
                =\frac{m_jg_j(t) + m_{j+1}g_{j+1}(t)}{m_j + m_{j+1}}. 
            \end{align*}
            Let $\bar P$ be the set of particles in the time-shifted solution which is the same as $P$, except that particles $p_j$ and $p_{j+1}$ are replaced with $\bar p$. $\bar P$ has $k$ particles, so, according to the inductive hypothesis, the trajectory of the center of mass of $\bar P$ for $t \le t_P-t_c$ ($t_P$ in the shifted time) is given by
            \begin{align}
                y_{\bar P}(t) &= \frac{\sum_{p_i \in \bar{P}}m_ig_i(t)}{\sum_{p_i \in \bar{P}}m_i}\nonumber 
                = \frac{\bar m g_{\bar p}(t) - m_jg_j(t) - m_{j+1}g_{j+1}(t) + \sum_{i=1}^{l}m_ig_i(t)}{\bar m - m_j - m_{j+1} + \sum_{i=1}^{l}m_i}\nonumber 
                = \frac{\sum_{i=1}^{l}m_ig_i(t)}{\sum_{i=1}^{l}m_i}\nonumber.
            \end{align}
            That is, in the original time reference, we have this equality for $t_c \le t\le t_P$. But since $P$ and $\bar P$ represent the same sets of particles after $\bar p$ has formed (for $t_c<t\le t_P$) their centers of mass are equal over this time period, so we have, for $t_c \le t \le t_P$, 
            \begin{align*}
                y_{P}(t) &= y_{\bar P}(t) = \frac{\sum_{i=1}^{l}m_ig_i(t)}{\sum_{i=1}^{l}m_i}. \qedhere
            \end{align*}
        \end{proof}
    We can now begin to understand collisions using simplified initial conditions. Using the previous lemma, we now prove that because the group's center of mass is equal to the center of mass of the ghost trajectories, it is also equal (at least until $t_P$) to the trajectory of a single ``center of mass" particle which replaces the entire group from the start. 
    \begin{lemma} \label{PC2}
        Take $P$ and $t_P$ as in Lemma \ref{PC1}. Consider a new $n-l+1$ particle  solution $\rho^*$ which has the same initial conditions as $\rho$, other than that the particles in $P$ have been replaced with a single particle  $p^*$ with initial mass, position, and velocity given by    
        \begin{align*}
            m_{p^*}=\sum_{i=1}^{l}m_i, \quad y_{p^*}(0)=\frac{\sum_{i=1}^{l}y_i(0)m_i}{\sum_{i=1}^{l}m_i}, \quad y'_{p^*}(0)=\frac{\sum_{i=1}^{l}y'_i(0)m_i}{\sum_{i=1}^{l}m_i},
        \end{align*}
        such that $p^*$ is initially located at the center of mass of $P$ with the same total mass and total momentum as $P$. Then, for $t \le t_P$, the trajectory of $p^*$ is equal to the trajectory of the center of mass of P: $y_{p^*}(t)=y_P(t)$.
    \end{lemma}
    \begin{proof}
        Let $k=\sum_{i=1-(q-1)}^{0}m_i-\sum_{i=l+1}^{n-(q-1)}m_i$. Then according to \eqref{kinematics} and the definition of $p^*$, 
        \begin{align*}
            y_{p^*}(t) 
            &=\frac{\sum_{i=1}^{l}m_iy_i(0)}{\sum_{i=1}^{l}m_i}\ +\ t\frac{\sum_{i=1}^{l}m_iy'_i(0)}{\sum_{i=1}^{l}m_i}\ +\ \frac{t^2}{4}k\frac{\sum_{i=1}^{l}m_i}{\sum_{i=1}^{l}m_i} 
            =\frac{\sum_{i=1}^{l}m_i\left(y_i(0)\ +\ y'_i(0)\ +\ \frac{t^2}{4}k\right)}{\sum_{i=1}^{l}m_i}.
        \end{align*}
        It is clear from the definition that the trajectories of particles in $P$ bound $y_{p*}(t)$. 
        Thus, by the assumption that particles in $P$ don't collide with other particles before $t_P$, and since $p^*$ affects outside particles the same way as $P$ (as can be observed from \eqref{kinematics}), we get that $p^*$ doesn't collide with other particles for $t<t_P$, so the above trajectory holds for $t \le t_P$ (we get $t=t_P$ by continuity of position). 
        Now, recalling \eqref{ghost-trajectory} and using Lemma \ref{PC1}, for $t \le t_P$ we have
        \begin{align*}
            y_P(t) &= \frac{\sum_{i=1}^{l}m_ig_i(t)}{\sum_{i=1}^{l}m_i} 
            =\frac{\sum_{i=1}^{l}m_i\left( y_i(0)\ +\ ty'_i(0)\ +\ \frac{t^2}{4}\left( k + \sum_{j=1}^{i-1}m_j-\sum_{j=i+1}^{l}m_j\right)\right)}{\sum_{i=1}^{l}m_i} \\[1em]
            &=\frac{\sum_{i=1}^{l}\left(m_iy_i(0)\ +\ tm_iy'_i(0)\ +\ \frac{t^2}{4}m_i k\right)+\frac{t^2}{4}\sum_{i=1}^{l}m_i\left(\sum_{j=1}^{i-1}m_j-\sum_{j=i+1}^{l}m_j\right)}{\sum_{i=1}^{l}m_i} \\[1em]
            &=\frac{\sum_{i=1}^{l}m_i\left(y_i(0)\ +\ ty'_i(0)\ +\ \frac{t^2}{4}k\right)}{\sum_{i=1}^{l}m_i} = y_{p^*}(t), 
        \end{align*}
        since $\sum_{i=1}^{l}m_i\left(\sum_{j=1}^{i-1}m_j-\sum_{j=i+1}^{l}m_j\right) = \sum_{i=1}^{l}\sum_{j=1}^{i-1}m_im_j-\sum_{i=1}^{l}\sum_{j=i+1}^{l}m_im_j=0$.
    \end{proof}

    We are now ready to prove the primary result of this section. The following lemma demonstrates that the replacement used in the previous lemma leaves the rest of the solution unchanged, and that when the particles in $P$ all combine at $t_P$, the particle they form has the same trajectory as that of $p^*$ at $t_P$, meaning that the combined and uncombined solutions proceed identically following $t_P$ (and therefore they are equivalent in terms of equilibrium).
    \begin{lemma} \label{PC}
        Let $P$ and $t_P$ as in Lemma \ref{PC1} and $\rho^*$ as in Lemma \ref{PC2}. Let $t_Q \le t_P$ be the least such time that $y_1(t_Q) = y_2(t_Q) = \ldots=y_l(t_Q)$. Then the solution $\rho$ is identical for all times $t \ge t_Q$ to $\rho^*$.
    \end{lemma}
    \begin{proof}
        We will prove this in two cases.
        
        {\textit{Case 1: $t_Q < t_P$.}}
        
        Let us examine the solutions at $t=t_Q$. The solutions are identical for particles outside of $P$, since for all $t<t_P$ the mass in $P$ and $p^*$ are equal and are on the same side of every other mass in the solutions. In $\rho$, $P$ has formed a single particle with mass equal to $m_{p^*}$ and trajectory $y_P(t)$, as the center of mass of the system is now just the position of the single particle. By Lemma \ref{PC2}, $y_{p^*}(t)=y_P(t)$ for all $t \le t_P$, and both functions are differentiable on $(0,t_P)$, so we have $y'_P(t_Q)=y'_{p^*}(t_Q)$, meaning that the newly formed particle has velocity at $t_Q$ equal to the velocity of $p^*$ at $t_Q$. Since $\rho$ and $\rho^*$ are identical mass distributions with (a.e.) identical velocity distributions at $t_Q$, existence and uniqueness of DSPS tells us that $\rho$ and $\rho^*$ must be identical for all times $t \ge t_Q$, as otherwise applying the semigroup property at $t_Q$ would yield two different solutions for the same initial condition.
        
        {\textit{Case 2: $t_Q = t_P$.}}
        
        Now we must account for the case in which collision with particles outside of $P$ occurs at the exact time the particles in $P$ first all combine. As before, the solutions are identical outside of $P$ for $t\le t_P$, so the incoming particles not in $P$ which collide with the particles in $P$ at $t_P$ in $\rho$ and collide with $p^*$ at $t_P$ in $\rho^*$ have the same masses and velocities. Clearly the mass and position of the particle formed by this collision is the same in $\rho$ and $\rho^*$, and clearly the solutions are identical outside of this collision at $t_Q$. We need to show that the velocity is the same in either case. Call the particles not in $P$ which collide with the particles in $P$ at $t_Q$ $\tilde p_{1},\ldots,\tilde p_{s}$. Call the distinct $r\leq l$ particles in $P$ just before the collision $\tilde p_{s+1}, \ldots, \tilde p_{s+r}$. This accounts for collisions that have happened in $P$ before $t_Q$. Note that $y_P(t)$ is still the location of the center of mass of $\tilde{p}_{s+1}$, $\ldots$, $\tilde{p}_{s+r}$. Then, conservation of momentum dictates that the velocity of the particle formed by this collision in $\rho$ at time $t_Q$ is
        \begin{align*}
            \frac{\sum_{i=1}^{s+r}\tilde m_i \frac{d^-}{dt}\tilde y_i(t)}{\sum_{i=1}^{s+r}\tilde m_i}.
        \end{align*}
        Now, noting that by Lemma \ref{PC1} we have
        \begin{align*}
            y_{p^*}(t) = y_P(t)=\frac{\sum_{i=s+1}^{s+r}\tilde m_i \tilde y_i(t)}{\sum_{i=s+1}^{s+r}\tilde m_i},
        \end{align*} 
        the velocity of the formed particle in $\rho^*$ is
        \begin{align*}
            \frac{\sum_{i=1}^{s}\tilde m_i \frac{d^-}{dt}\tilde y_i(t)\ +\ m_{p^*}\frac{d^-}{dt}y_{p^*}(t)}{\sum_{i=1}^{s}\tilde m_i\ +\ m_{p*}}
            &=\frac{\sum_{i=1}^{s}\tilde m_i \frac{d^-}{dt}\tilde y_i(t)\ +\ \Big(\sum_{i=s+1}^{s+r} \tilde m_i\Big)\frac{d^-}{dt}\Bigg(\frac{\sum_{i=s+1}^{s+r}\tilde m_i \tilde y_i(t)}{\sum_{i=s+1}^{s+r}\tilde m_i}\Bigg)}{\sum_{i=1}^{s}\tilde m_i
            \ +\ \sum_{i=s+1}^{s+r}\tilde m_i} 
            \\[1em]
            &=\frac{\sum_{i=1}^{s+r}\tilde m_i \frac{d^-}{dt}\tilde y_i(t)}{\sum_{i=1}^{s+r}\tilde m_i}.
        \end{align*}
    Thus, in both $\rho$ and $\rho^*$, at time $t_Q$ the mass and velocity distributions are identical, so just as in case 1 they must be so for all $t \ge t_Q$.
    \end{proof}
    This property is greatly useful; for example, if two particles collide at $t_c$ before colliding with other particles, we can take $P$ to be the two particles and $t_Q = t_P = t_c$ to obtain that the solution in which we initially replace them by the particle at their center of mass with their total mass and mass averaged velocity is identical to the original solution for $t \ge t_c$. This leads to a neat way to calculate whether a system goes to equilibrium: we can repeatedly check which pair of particles will collide first,\footnote{Proposition \ref{properties} can show that this will be the adjacent pair which minimizes the quantity \mbox{$\big(y'_i(0) - y'_{i+1}(0) - \sqrt{(y'_i(0) - y'_{i+1}(0))^2 - (y_{i+1}(0) - y_i(0)(m_i + m_{i+1})}\big)/\big(m_i + m_{i+1}\big)$} such that it is nonnegative.} examine the solution in which these are combined, then check the first collision in this solution and repeat. If we eventually reach a multi-particle solution with no collisions, the original solution is divergent. If we eventually reach a single particle solution, the original solution goes to equilibrium as the previous lemma tells us that it is identical after some time to a single particle.
 
    \subsection{The Hamiltonian}
        \begin{definition}
            The Hamiltonian of a GSPS $(\rho, v)$ is the function 
            \begin{equation}\label{hamiltonian}
                H(\rho_t,v_t)=\|v_t\|_{L^2(\rho_t)}^2-\frac{1}{2}\iint|x-y|\rho_t(dx)\rho_t(dy).
            \end{equation}
        \end{definition}

        The Hamiltonian is our notion of energy for the system. ($||v_t||_{L^2(\rho_t)}^2$ may be thought of as kinetic energy, and the second part as potential.) We will demonstrate that it has the properties one would expect the energy to have. We will also examine its connection to a special type of collisions defined below; these are a phenomenon unique to the repulsive case, as particles in the attractive PEP or PE equations do not exhibit the behavior of slowing down as they approach each other. 
    
        \begin{definition}\label{glancing-perfect} For a DSPS, a collision between particles is called \textbf{glancing} if the trajectories which intersect are tangent at the point of intersection. We call a finite DSPS \textbf{perfect} if it goes to an equilibrium and all collisions between particles are glancing.\end{definition}
        
        Note that a collision is glancing if and only if the trajectories of all the participating particles are differentiable at the collision. See figure \ref{fig:glancing-collisions}. 

        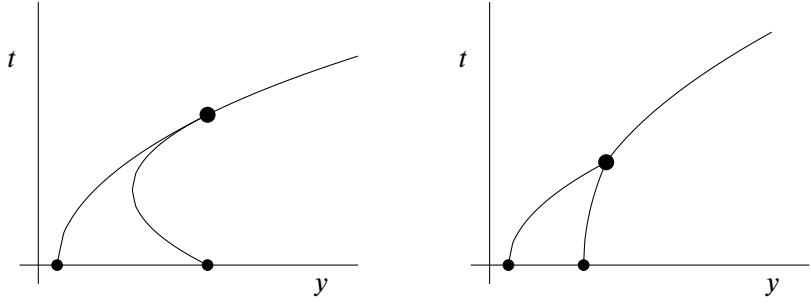
\begin{figure}
            \centering
            \begin{tikzpicture}[scale=1]
                \draw (-1.5-3,0)--(3-3,0);
                \draw (-1.25-3,-0.25)--(-1.25-3,3.5);
                \draw (-1.6-3,2.75) node {$t$};
                \draw (2.5-3,-0.3) node {$y$};
            
                \draw plot[smooth, domain=0-3:1-3] (\x, {sqrt(\x+3)+1});
                \draw plot[smooth, domain=0-3:1-3] (\x, {-sqrt(\x+3)+1});
    
                \draw plot[smooth, domain=-1-3:1-3] (\x, {1.414*sqrt(\x+3+1)});
    
                \draw plot[smooth, domain=1-3:3-3] (\x, {0.5+0.5*sqrt(\x+3)+0.7072*sqrt(\x+3+1)});
    
                \draw[fill] (-1-3,0) circle [radius=2pt];
                \draw[fill] (1-3,0) circle [radius=2pt];
                \draw[fill] (1-3,2) circle [radius=2.828pt];
    
                \*\*\*\*\*\*\*\*\*\*\*\*\*\*\*
    
                \draw (-1.5+3,0)--(3+3,0);
                \draw (-1.25+3,-0.25)--(-1.25+3,3.5);
                \draw (-1.6+3,2.75) node {$t$};
                \draw (2.5+3,-0.3) node {$y$};
            
                \draw plot[smooth, domain=-1+3:3.29938] (\x, {1.2*sqrt(\x-3+1)});
                \draw plot[smooth, domain=0+3:3.29938] (\x, {2.5*sqrt(\x-3)});
    
                \draw plot[smooth, domain=3.29938:2.5+3] (\x, {0.6*sqrt(\x-3+1)+1.25*sqrt(\x-3)});
        
                \draw[fill] (-1+3,0) circle [radius=2pt];
                \draw[fill] (3,0) circle [radius=2pt];
                \draw[fill] (3.29938, 1.36788) circle [radius=2.828pt];
            \end{tikzpicture}
            \caption{Glancing and non-glancing sticky collisions.}
            \label{fig:glancing-collisions}
        \end{figure}
        
        \begin{theorem}\label{energy-constant}
            The Hamiltonian of a DSPS to \eqref{PEP} is piecewise constant, with discontinuities only possibly occurring where non-glancing collisions occur. 
        \end{theorem}
        \begin{proof}
            Let $\rho_t=\sum_im_i\delta_{y_i(t)}$. Formally, consider the right derivative of the Hamiltonian. 
            \begin{align}\nonumber
                \frac{d^+}{dt}H(\rho_t,v_t) =& \frac{d^+}{dt}\sum_{i}m_i\big(y_i'(t)\big)^2-\frac{d^+}{dt}\sum_{i}m_i\sum_{\{j|y_j(t)<y_i(t)\}}m_j\big(y_i(t)-y_j(t)\big), 
            \end{align}
            where here, $y_i'$ denotes $\frac{d^{+}}{dt}y_i$. Since we are taking the right derivative, we can equate the acceleration of each trajectory $y_i$ to $\big(\sum_{\{j|y_j<y_i\}}m_j-\sum_{\{j|y_j>y_i\}}m_j\big)/2$. Noting the absolute convergence of each sum,
            \begin{align}
                \frac{d^+}{dt}H(\rho,v) =& \sum_{i}2m_iy_i'a_i-\sum_{i}m_i\sum_{\{j|y_j<y_i\}}m_j(y_i'-y_j')\nonumber
                \\
                =& \sum_{i}m_i\bigg(y_i'\big(\sum_{\{j|y_j<y_i\}}m_j\ -\sum_{\{j|y_j>y_i\}}m_j\big)-y_i'\sum_{\{j|y_j<y_i\}}m_j\ +\sum_{\{j|y_j<y_i\}}m_jy_j'\bigg)\nonumber
                \\
                =& \sum_{i}m_i\bigg(-y_i'\sum_{\{j|y_j>y_i\}}m_j\ +\sum_{\{j|y_j<y_i\}}m_jy_j'
                \bigg).\nonumber
            \end{align}
            Changing indices, justified by the finite momentum of the solution, 
            \begin{align}\nonumber
                \frac{d^+}{dt}H(\rho_t,v_t)=-\sum_{j}m_j\sum_{\{i|y_i<y_j\}}m_iy_i'(t)\ +\ \sum_{i}m_i\sum_{\{j|y_j<y_i\}}m_jy_j'(t)=0. 
            \end{align}
            Therefore, at each time $t$, either the derivative of $H(\rho_t,v_t)$ is $0$, in which case $H(\rho_t,v_t)$ is constant, or $H(\rho_t,v_t)$ is not differentiable at time $t$; however, this can only happen if $y_i(t)$ is not differentiable for some $i$, meaning there is a non-glancing collision!
        \end{proof}
        
        It is clear that $H(\delta_0,0)=0$. Furthermore, it is easy to show that given $(\rho,v)$, if we take $(\bar{\rho},\bar{v})$ as in Proposition \ref{properties}, \ref{galilean-invariance}, then $H(\rho_t,v_t)=H(\bar{\rho}_t,\bar{v}_t)+m_0^2$. Since we have now proved that $H(\rho_t,v_t)$ is constant both between collisions and at glancing collisions, we obtain the following corollary.\footnote{Finite DSPS have finitely many collisions, and thus can collapse to equilibrium only in finite time.}
        
        \begin{corollary}\label{energy-perfect}
            Let $(\rho,v)$ be a finite DSPS. If $(\rho,v)$ is perfect, then $H(\rho_t,v_t)=m_0^2$ for all $t\geq0$. 
        \end{corollary}
        \begin{lemma}\label{PE-continuity}
            The potential energy, 
            \begin{align}
                U(\rho_t)=-\frac{1}{2}\iint|x-y|\rho_t(dx)\rho_t(dy), 
            \end{align}
            is everywhere continuous in $t$ for any finite DSPS $(\rho,v)$ to the repulsive PEP equations. 
        \end{lemma}
        \begin{proof}
            Let $\varepsilon>0$. Fix $t_0\geq0$ and consider a difference in potential energy for a discrete finite solution $\rho_t=\sum_{i=1}^n m_i\delta_{y_i(t)}$, $\sum_{i=1}^nm_i=1$, with $\{y_i(t)\}_{i=1}^n$ an increasing sequence for any $t$. By the continuity of each $y_i$, $|y_i(t_0)-y_i(t)|<\varepsilon/2$ when $|t_0-t|<\eta_i$ for some $\eta_i>0$. Take $\eta=\min\{\eta_i\}_{i=1}^n$. For $|t_0-t|<\eta$,
            \begin{align*}
                \big|U(\rho_t)-U(\rho_{t_0})\big|=\ &\Big|\sum_{i=1}^n m_i\sum_{j<i}m_j\big(y_i(t_0)-y_j(t_0)\big)-\sum_{i=1}^n m_i\sum_{j<i}m_j\big(y_i(t)-y_j(t)\big)\Big|
                \\
                =\ &\Big|\sum_{i=1}^n m_i\sum_{j<i}m_j\big(y_i(t_0)-y_i(t)\big)+\sum_{i=1}^n m_i\sum_{j<i}m_j\big(y_j(t)-y_j(t_0)\big)\Big|
                \\
                \leq\ &\sum_{i=1}^n m_i\sum_{j<i}m_j\big|y_i(t_0)-y_i(t)\big|+\sum_{i=1}^n m_i\sum_{j<i}m_j\big|y_j(t_0)-y_j(t)\big|
                \\
                <\ &\frac{\varepsilon}{2}\sum_{i=1}^nm_i\sum_{j<i}m_j+\frac{\varepsilon}{2}\sum_{i=1}^nm_i\sum_{j<i}m_j\leq\ \varepsilon\sum_{i=1}^nm_i\sum_{j=1}^nm_j=\varepsilon. \qedhere
            \end{align*}
        \end{proof}
        \begin{proposition}\label{KE-discontinuity}
            For any finite DSPS to the repulsive PEP system, at collisions between particles, the kinetic energy $\|v_t\|_{L^2(\rho_t)}^2$ does not increase (in the sense that $\lim_{t\to {t_0}^-}\|v_t\|^2_{L^2(\rho_t)}\geq\lim_{t\to {t_0}^+}\|v_t\|^2_{L^2(\rho_t)}$), and remains constant if and only if the collision is glancing. That is, at non-glancing collisions, $\|v_t\|_{L^2(\rho_t)}^2$ exhibits a decreasing jump discontinuity. 
        \end{proposition}
        \begin{proof}
            Let $t_0$ be any time of collision, and suppose the particles of trajectories $\{y_i(t)\}_{i=1}^n$ and masses $\{m_i\}_{i=1}^n$ collide to form particles of trajectories $\{Y_{A_j}(t)\}_{j=1}^N$ where $\{A_j\}_{j=1}^N$ are disjoint and $\bigcup_{j=1}^N A_j=\{1,2,\ldots,n\}$. Consider the kinetic energy at $t_0$. Without loss of generality, suppose that $y_1(t)\leq y_2(t)\leq\ldots\leq y_n(t)$ for all $t<t_0$. Throughout this proof, denote $y_i'=\frac{d^{-}}{dt}y_i$.\footnote{Because there are only finitely many collisions, $y_i$ is left- and right-differentiable.}
            \begin{align*}
                \lim_{t\to t_0^{\ -}}\|v_t\|_{L^2(\rho_t)}^2=\sum_{i=1}^n m_i\big(y_i'(t_0)\big)^2
                =\sum_{j=1}^N\sum_{i\in A_j}m_i\big(y_i'(t_0)\big)^2,
            \end{align*}
            and, appealing to the conservation of momentum,  
            \begin{align*}
                \lim_{t\to t_0^{\ +}}\|v_t\|_{L^2(\rho_t)}^2=&\sum_{j=1}^N\Big(\sum_{i\in{A_j}}m_i\Big)\Big(\frac{d^+Y_{A_j}}{dt}(t_0)\Big)^2= \sum_{j=1}^N\frac{\Big(\sum_{i\in A_j}m_iy_i'(t_0)\Big)^2}{\sum_{i\in A_j}m_i}. 
            \end{align*}
            The Cauchy-Schwarz inequality applied to $\{\sqrt{m_i}y_i'(t_0)\}_{i\in A_j}$ and $\{\sqrt{m_i}\}_{i\in A_j}$ gives  
            \begin{align*}
                \Big(\sum_{i\in A_j}m_iy_i'(t_0)\Big)^2\leq\bigg(\sum_{i\in A_j}m_i\big(y_i'(t_0)\big)^2\bigg)\sum_{i\in A_j}m_i.
            \end{align*}
            Therefore, $\lim_{t\to t_0^{\ -}}\|v_t\|_{L^2(\rho_t)}^2\geq\lim_{t\to t_0^{\ +}}\|v_t\|_{L^2(\rho_t)}^2$, as desired. 
    
            Note that equality holds in the Cauchy-Schwarz inequality if and only if one vector is a scalar multiple of another; that is, if and only if $\{\sqrt{m_i}y'_i(t_0)\}=\{k\sqrt{m_i}\}$ for all $i$. Therefore, we have equality if and only if all velocities are the same, i.e., the collision is glancing. 
        \end{proof}
        \begin{theorem}\label{DSPS-Hamiltonian}
            The Hamiltonian of any finite DSPS $(\rho,v)$, $\rho\in\mathcal{P}_2(\mathbb{R})$, $v\in L^2(\rho)$, to the repulsive PEP equations is nonincreasing as a function of time. 
        \end{theorem}
        \begin{proof} 
            By Theorem \ref{energy-constant}, the Hamiltonian is constant between non-glancing collisions (so it is trivially nonincreasing). Therefore, it remains to show the Hamiltonian is nonincreasing at non-glancing collisions. 
            
            Consider the potential energy, $U(\rho_t)$. In the case of a non-glancing collision, Proposition \ref{KE-discontinuity} tells us that $\|v_t\|_{L^2(\rho_t)}^2$ exhibits a decreasing jump discontinuity. Assume for contradiction that $H(\rho_t,v_t)=\|v_t\|_{L^2(\rho_t)}^2+U(\rho_t)$ increases at the time of a non-glancing collision, $t_0$. Since we are considering a finite number of nonzero masses, we have some $\varepsilon>0$ such that there are no other collisions in $(t_0-\varepsilon,t_0+\varepsilon)$, so the kinetic energy is continuous other than the jump discontinuity at $t_0$. 
            However, in order for the Hamiltonian to increase, potential energy then needs to exhibit an increasing jump discontinuity, violating its continuity; thus, it is not possible. 
        \end{proof}
        The results in this section allow us to reframe the notion of perfect solutions in the case of finite DSPS, and thus, extend our notion of a perfect solution beyond the discrete case: 
        \begin{definition}\label{perfect-generalization}
            We call a GSPS $(\rho_t,v_t)$ with $0$ total momentum to the repulsive PEP system \textbf{perfect} if it tends to an equilibrium and for all $t\geq0$, $t\mapsto H(\rho_t,v_t)$ is constant. 
        \end{definition}
        For a general GSPS, we can say that $(\rho,v)$ is perfect so long as $(\bar{\rho},\bar{v})$ in the sense of Proposition \ref{properties}, \ref{galilean-invariance} is perfect. In the case of finite time collapse, note that the constant above is the square of the total momentum of the solution.\footnote{In the case of asymptotic equilibrium, one would expect that, since (as we will see) $U(\rho_t)\to0$, $||v_t||_{L^2(\rho_t)}\to||v_\infty||_{L^2(\rho_\infty)}$, i.e., if a solution collapses to a stationary mass, the kinetic energy goes to 0. However, we allow also for the possibility that some solutions may exhibit oscillatory behavior such that $W(\rho_t,\delta_0)\to0$, but $||v_t||_{L^2(\rho_t)}>\varepsilon>0$ for all $t\geq0$. Such an example is yet to be constructed.}  
        
        To extend the results of the Hamiltonian to continuous distributions, we will need the following framework for approximating GSPS. 
    \subsection{Strong Convergence of Discrete Approximations to GSPS}
        Given our wealth of tools for studying discrete measures, it is natural that we should want some mechanism to extend discrete results to more general measures.
        The following theorem gives the conditions under which a series of discrete approximations converge appropriately to a more general measure. The proof makes use of our results on the Hamiltonian, as well as crucial results from Proposition 3.3 in \cite{Tudorascu-2008}, in order to extract appropriate bounds, approximations, and subsequences to obtain the result. 
    
        \begin{theorem} Let $\rho_0\in\mathcal{P}_2(\mathbb{R}), v_0\in C(\mathbb{R})$ of at most quadratic growth and consider a sequence of discrete approximations $\{\rho_0^n\}_n$ to $\rho_0$ as in Proposition 3.3 in \cite{Tudorascu-2008}. Let $(\rho,v)$ be the GSPS corresponding to $(\rho_0,v_0)$ as constructed in the proof of Proposition 3.3 \cite{Tudorascu-2008} based on the approximations $\{\rho_0^n\}_n$.  Then, for any $T>0$, $\max_{t\in[0,T]}W_2(\rho_t^n,\rho_t)\to0$ and, possibly up to a subsequence, $\{v_t^{n}\rho_t^{n}\}_n$ converges strongly to $v_t\rho_t$ for a.e. $t\geq0$ in the sense that 
            \begin{align}
                \int v_t^n\varphi d\rho_t^n\to\int v_t\varphi d\rho_t
            \end{align}
            for all $\varphi\in C_c^{\infty}(\mathbb{R})$ and
            \begin{align}
                \int|v_t^n|^2d\rho_t^n\to\int|v_t|^2d\rho_t
            \end{align}
            for a.e. $t\geq0$. 
        \end{theorem}
        \begin{proof}
            By Theorem \ref{DSPS-Hamiltonian}, $t\mapsto H(\rho_t^n,v_t^n)$ is nonincreasing for $t\in[0,T]$. Thus, we can establish that
            \begin{align}\label{HamiltonianBound}
                H(\rho_t^n,v_t^n)\leq H(\rho_0^n,v_0^n)\leq C<\infty 
            \end{align}
            for all $n\in\mathbb{N}$, $t\geq0$ by the bounds from Proposition 3.3 in \cite{Tudorascu-2008}. By this proposition, $\max_{t\in[0,T]}W_2(\rho_t^n,\rho_t)\to0$, so we can deduce that
            \begin{align*}
                -2U(\rho_t^n)=&\iint|x-y|\rho_t^n(dx)\rho_t^n(dy)\ \ \leq \ \ 2\|N_t^n\|_{L^1(0,1)}
                \\
                \ \ \leq& \ \ 2\|N_t\|_{L^1(0,1)}+2\|N_t^n-N_t\|_{L^1(0,1)}
                \ \ \leq\ \ \ 2\|N_t\|_{L^2(0,1)}+2W_2(\rho_t^n,\rho_t)
                \\
                \ \ \leq& \ \ 2W_2(\rho_t^n,\rho_t)+2\sqrt{\int y^2\rho_t(dy)}       
                \ \ \leq \ \ 2W_2(\rho_t^n,\rho_t)+2\max_{t\in[0,T]}\sqrt{\int y^2\rho_t(dy)}
            \end{align*}
            since the function $t\mapsto\int y^2\rho_t(dy)$, $t\in[0,T]$, is continuous. Thus, $-2U(\rho_t^n)$ is bounded uniformly with respect to $n\in\mathbb{N}$ and $t\in[0,T]$. 

            Now, we deduce from \eqref{HamiltonianBound} that 
            \begin{equation}\label{velo-l2-bound}
                \|v_t^n\|_{L^2(\rho_t^n)}\leq C_T<\infty\ \mbox{ for all }\ t\in[0,T].
                \end{equation}
                As a first consequence, we have:
            \begin{align*}
                \int_0^T \!\! \int_0^1|\dot{N}_t^n(x)|^2dxdt \ \leq \ TC_T^2.
            \end{align*}
            Since 
            \begin{align*}
                \int_0^T \!\! \int_0^1|N_t^n(x)-N_t(x)|^2dxdt \ \leq \ T\max_{0\in[0,T]}W_2^2(\rho_t^n,\rho_t),
            \end{align*}
            we deduce that $N^n\to N$ in $L^2((0,T)\times (0,1))$ and $\dot{N^n}\rightharpoonup\dot{N}$ weakly in $L^2((0,T)\times (0,1))$. 

            Proposition 3.3, \cite{Tudorascu-2008} also shows that the right-continuous cumulative distribution functions $M_t^n$ and $M_t$ satisfy 
            \begin{align}\label{CDF-condition}
                \max_{t\in[0,T]}\|M_t^n-M_t\|_{L^1(\mathbb{R})}=\max_{t\in[0,T]}W_1(\rho_t^n,\rho_t)\to0
            \end{align}
            as $n\to\infty$. Furthermore, in the proof of Proposition 3.3, \cite{Tudorascu-2008} it is shown that $v_t^n\rho_t^n=\partial_y[{F}_n(t,M^n(t,y))]$ and $v_t\rho_t=\partial_y[{F}(t,M(t,y))]$, where 
            \begin{align*}
                {F}_n(t,m)=\int_0^m v_0\circ N_0^n(w)dw+t\int_0^ma_n(w)dw
            \end{align*}
            and
            \begin{align*}
                {F}(t,m)=\int_0^m v_0\circ N_0 (w)dw+t\int_0^ma(w)dw.
            \end{align*}
            Here, $a$ and $a_n$ are the acceleration terms.
            It is easy to see that ${F}_n\to{F}$ uniformly on $[0,T]\times[0,1]$ as uniformly continuous functions. 
            Since $\max_{t\in[0,T]}\|M^n(t,\cdot)-M(t,\cdot)\|_{L^1(\mathbb{R})}$ converges to zero, we get $$\|M^n-M\|_{L^1((0,T)\times\mathbb{R})}\rightarrow0\mbox{ as }n\rightarrow\infty,$$ 
            so we can extract a subsequence $\{M^{n_k}\}_k$ such that $M^{n_k}\rightarrow M$ Lebesgue a.e. in $(0,T)\times\mathbb{R}$. This implies that for a.e. $t\in(0,T)$, the sequence $\{F(t,M^{n_k}(t,\cdot))\}_k$ converges to $F(t,M(t,\cdot))$ a.e. on $\mathbb{R}$. By \eqref{velo-l2-bound} applied to $t=0$ and the uniform bounds on $a_n$ and $a$, we see that $\{F(t,M^{n_k}(t,\cdot)\}_k$ and $F(t,M(t,\cdot))$ are uniformly bounded in $L^\infty(\mathbb{R})$.
            Along with the uniform convergence of $F_n$ to $F$, we infer ${F}_{n_k}(t,M^{n_k}(t,\cdot))-{F}(t,M(t,\cdot))$ converges to zero in $L^1_{loc}(\mathbb{R})$. Since the limit does not depend on the extracted subsequence, we get that the full sequence ${F}_{n}(t,M^{n}(t,\cdot))-{F}(t,M(t,\cdot))$ converges to zero in $L^1_{loc}(\mathbb{R})$. Thus, $\partial_y\big[{F}_n(t,M^n(t,\cdot))\big]\rightharpoonup\partial_y\big[{F}(t,M(t,\cdot))\big]$, i.e. $v_t^n\rho_t^n\rightharpoonup v_t\rho_t$ as distributions as $n\to\infty$ for a.e. $t\in[0,T]$.  In the proof of Proposition 3.3, \cite{Tudorascu-2008} it is shown that for all $t\in[0,T]$ there exists an increasing, convex, superlinear function $\zeta_t$ on $[0,\infty)$ with $\zeta_t(0)=0$ and such that 
            $$\sup_{n\geq1}\int\zeta_t(y^2)\rho^n_t(dy)<\infty.$$
            This implies the sequence $\{\rho^n_t\}_n$ has uniformly (with respect to $n$) integrable second moments, so we can use \eqref{velo-l2-bound} again to conclude, via a standard approximation argument, that
            \begin{align}\label{yv^n-convergence}
                \lim_{n\to\infty}\int yv_t^n(y)\rho_t^n(dy)=\int yv_t(y)\rho_t(dy)\ \mbox{ for all }\ t\in[0,T].
            \end{align}
            We also infer that we can use test functions of linear growth in the weak formulation of the two coupled equations; see \cite{Suder-Tudorascu} for more details on how to extend the class of test functions from smooth and compactly supported to smooth and at most linear. 
            The momentum equation \eqref{momentumn-equation} for $\varphi(y)=y$ yields  
            \begin{align*}
                \frac{d}{dt}\int yv_t^n(y)\rho_t^n(dy)=\int|v_t^n|^2d\rho_t^n-\iint y\sgn(y-z)\rho_t^n(dz)\rho_t^n(dy),
            \end{align*}
            so 
            \begin{align}\label{P_n}
                \int \! yv_t^n(y)&\rho_t^n(dy)-\int \! yv_0^n(y)\rho_0^n(dy)
                = \int_0^t \!\! \int \! |v_s^n|^2d\rho_s^n ds-\int_0^t \!\! \iint \! y\sgn(y-z)\rho_s^n(dz)\rho_s^n(dy)ds
            \end{align}
            for all $t\in[0,1]$. Likewise, since $(\rho,v)$ solves the IVP for initial data $(\rho_0,v_0)$, we deduce 
            \begin{align}\label{P}
                \int \! yv_t(y)\rho_t(dy)&-\int \! yv_0(y)\rho_0(dy)
                =\int_0^t \!\! \int \! |v_s|^2d\rho_sds-\int_0^t \!\! \iint \! y\sgn(y-z)\rho_s(dz)\rho_s(dy)ds. 
            \end{align}
            We use \eqref{yv^n-convergence} to deduce that the left-hand side of \eqref{P_n} converges to the left-hand side of \eqref{P} for all $t\in[0,T]$. Likewise, $\max_{t\in[0,T]}W_2(\rho^n_t,\rho_t)\rightarrow0$ shows that the second term in the right-hand side of \eqref{P_n} converges to its counterpart from \eqref{P}. This gives us, for all $t\in[0,T]$,
            \begin{align*}
                \lim_{n\to\infty}\int_0^t \!\! \int|v_s^n|^2d\rho_s^nds=\int_0^t \!\! \int|v_s|^2d\rho_sds. 
            \end{align*}
            In particular, as $n\to\infty$,
            \begin{align*}
                \int_0^T \!\! \int_0^1|\dot{N^n}(x,t)|^2dxdt\to\int_0^T \!\! \int_0^1|\dot{N}(x,t)|^2dxdt. 
            \end{align*}
            We have shown that $\dot{N}^n\rightharpoonup\dot{N}$ weakly; it follows that $\dot{N}^n\to\dot{N}$ strongly in $L^2((0,T)\times(0,1))$. In particular, up to a subsequence $\{n_k\}_k$, we have 
            \begin{align*}
                \int_0^1|\dot{N}^{n_k}(x,t)-\dot{N}(x,t)|^2dx\to0
            \end{align*}
            for a.e. $t\in[0,T]$. Thus, for a.e. $t\in[0,T]$
            \begin{align*}
                \int|v_t^{n_k}|^2d\rho_t^{n_k}&=\int_0^1|\dot{N}^{n_k}(x,t)|^2dx\to\int_0^1|\dot{N}(x,t)|^2dx=\int|v_t|^2d\rho_t. \qedhere
            \end{align*}
        \end{proof}
        Considering the above subsequence gives us the following. 
        \begin{corollary}\label{existence-general}
            Let $\rho_0\in\mathcal{P}_2(\mathbb{R})$, and $v_0\in C(\mathbb{R})$ of at most quadratic growth. Then for any GSPS $(\rho,v)$ originating at $(\rho_0,v_0)$ there exists a sequence of discrete approximations $(\rho_0^n,v_0^n)$ to the initial data satisfying the conditions from Proposition 3.3 \cite{Tudorascu-2008} for which the corresponding repulsive SPS solutions $(\rho^n,v^n)$ satisfy, for all $T>0$, 
            \begin{enumerate}[label={(\arabic*)}]
                \item$\max_{t\in[0,T]}W_2(\rho_t^n,\rho_t)\to0$ as $n\to\infty$, and 
                \item $v_t^n\rho_t^n\to v_t\rho_t$ strongly as $n\to\infty$ for a.e. $t\in[0,T]$. Notably, we have that for a.e. $t\in[0,T]$, $H(\rho_t^n,v_t^n)\to H(\rho_t,v_t)$ as $n\to\infty$. 
            \end{enumerate}
        \end{corollary}
        Note how this allows us to extend Theorem \ref{DSPS-Hamiltonian} to the general case. We state this observation in the following theorem. 
        \begin{theorem}\label{GSPS-Hamiltonian}
            Let $\rho_0\in\mathcal{P}_2(\mathbb{R})$, $v_0\in L^2(\rho_0)$. Then the Hamiltonian of any GSPS originating at $(\rho_0,v_0)$ is nonincreasing as a function of time. 
        \end{theorem}
        \begin{corollary}
            Let $\rho_0\in\mathcal{P}_2(\mathbb{R})$, $v_0\in L^2(\rho_0)$ and $(\rho,v)$ be a GSPS to the repulsive PEP equations starting at $(\rho_0,v_0)$. If $(\rho,v)$ tends to an equilibrium, then for all $t\geq0$, 
            \begin{align*}
                H(\rho_t,v_t)\geq\Big(\int v_0d\rho_0\Big)^2.
            \end{align*}
            In particular, the Hamiltonian of a solution which goes to equilibrium is never negative. 
        \end{corollary}
        \begin{proof}
            First, take $(\bar{\rho},\bar{v})$ as in Proposition \ref{properties}, \ref{galilean-invariance}. We then have 
            \begin{align*}
                H(\rho_t,v_t)&\ =\ \|v_t\|_{L^2(\rho_t)}^2+U(\rho_t)\ =\ \|\bar{v}_t\|_{L^2(\bar{\rho}_t)}^2+m_0^2+U(\bar{\rho}_t)
                \ =\ H(\bar{\rho}_t,\bar{v}_t)+\Big(\int v_0d\rho_0\Big)^2.
            \end{align*}
            Therefore, we only need to show that $H(\bar{\rho}_t,\bar{v}_t)\geq0$. First, suppose $(\bar{\rho}_t,\bar{v}_t)$ goes to equilibrium in finite time. It is clear that at equilibrium, we have $H(\rho_t,v_t)=0$. Since $H(\rho_t,v_t)$ is nonincreasing in time (Theorem \ref{GSPS-Hamiltonian}), $H(\rho_t,v_t)<0$ is impossible. 

            Now consider the case of asymptotic equilibrium. Noting that the only possibility (Proposition \ref{properties}, \ref{only-equilibrium}) is $\rho_\infty=\delta_0$, say that $W_1(\rho_t,\delta_0)\to0$ as $t\to\infty$. Then, the potential energy goes to 0: 
            \begin{align*}
                \frac{1}{2}\iint|x-y|&\rho_t(dx)\rho_t(dy)=\frac{1}{2}\int_0^1 \!\! \int_0^1|N_t(x)-N_t(y)|dxdy
                \\
                \leq&\frac{1}{2}\int_0^1|N_t(x)|dx+\frac{1}{2}\int_0^1|N_t(y)|dy= W_1(\rho_t,\delta_0).
            \end{align*}
            Now, assume for contradiction that for some $t_0\geq0$, $H(\rho_{t_0},v_{t_0})<0$, meaning that for all $t\geq t_0$, (since the Hamiltonian is nonincreasing by Theorem \ref{GSPS-Hamiltonian})
            \begin{align*}
                \|v_t\|_{L^2(\rho_t)}^2<\frac{1}{2}\iint|x-y|\rho_t(dx)\rho_t(dy)\leq W_1(\rho_t,\delta_0).
            \end{align*}
            Therefore, both potential and kinetic energy go to $0$, meaning $H(\rho_t,v_t)\to0$. Since $H(\rho_t,v_t)$ is nonincreasing, it cannot be negative for any time, a contradiction.         
        \end{proof}
                
\section{Perfect Solutions}\label{perfect-section}
        The following section is dedicated to further exploring the phenomenon of perfect solutions; those in which all collisions are glancing, and equilibrium is reached. We find it prudent before we begin to make a remark that will be of help throughout the section.
        \begin{remark}\label{lipschitz}
            In a perfect DSPS, the velocity of every particle is Lipschitz continuous. 
        \end{remark}
        The proof of this statement is a single line: its derivative (acceleration) exists almost everywhere and is bounded above and below by $\pm 1/2$. The usefulness of this statement resides primarily in the fact that Lipchitz continuity allows us to apply the fundamental theorem of calculus. Since position is clearly $C^1$ which is stronger, we have for any particle trajectory $y(t)$ in a perfect solution that
        \begin{align}\label{position_double_integral}
            y(t) = y(0) + ty'(0) + \int^t_0 \!\! \int^x_0y''(w) \ dw \ dx.
        \end{align}
        This will prove useful ahead. 

        We now turn our attention to the primary result of the section: the existence and uniqueness of velocity distributions yielding perfect solutions given discrete initial mass distributions. Proof of this fact will require several lemmas, the first of which has much merit on its own to the understanding of glancing collisions. It shows that whether a first collision of two particles is glancing depends solely on the interplay between their masses and their relative initial positions and velocities, or equivalently the time of their collision. The argument is mainly algebraic. 
        \begin{lemma}\label{perfect-adjacent-particles}
            Let $\rho$ be a discrete solution, $\rho_t=\sum_i m_i\delta_{y_i(t)}$, with $y_1(t)\leq y_2(t)\leq\ldots\leq y_n(t)$ for all $t\geq0$. Consider two adjacent particles with masses $m_i$, $m_{i+1}$ and trajectories $y_i(t)$, $y_{i+1}(t)$ which collide with each other but do not collide with other particles beforehand. Then the following statements are equivalent:
           
            (i) The particles collide glancingly.

            (ii) The particles initial positions and velocities satisfy
            \begin{equation}\label{veloc-perfect}
                y_i'(0)-y_{i+1}'(0)=\sqrt{\big(y_{i+1}(0)-y_i(0)\big)(m_i+m_{i+1})}.
            \end{equation}

            (iii) The particles collide at time $t_c$, where 
            \begin{align}\label{TC}
                t_c=2\sqrt{\frac{y_{i+1}(0)-y_i(0)}{m_i+m_{i+1}}}.
            \end{align}
        \end{lemma}
    
        \begin{proof}
            Let $t_0$ be the time of collision, and let $k = \sum_{j < i}{m_j} - \sum_{j > i + 1}{m_j}$. Then, while $t < t_0$, we have 
            \begin{align}
                y_i''(t) &= \frac{k - m_{i+1}}{2}, \ 
                y_i'(t) = y_i'(0) + \frac{k - m_{i+1}}{2} t, \label{vel__reference_i}\\
                y_i(t) &= y_i(0) + y_i'(0)t + \frac{k - m_{i+1}}{4} t^2, \label{pos__reference_i}\\
                y_{i+1}''(t) &= \frac{k + m_{i}}{2},\ 
                y_{i+1}'(t) = y_{i+1}'(0) + \frac{k + m_{i}}{2} t, \label{vel__reference_i+1}\\
                y_{i+1}(t) &= y_{i+1}(0) + y_{i+1}'(0) t + \frac{k + m_{i}}{4} t^2. \label{pos__reference_i+1}
            \end{align}
            Note that \eqref{pos__reference_i} and \eqref{pos__reference_i+1} apply also at $t=t_0$ due to the continuity of position. Now, assume $i$; that the particles collide glancingly. Since the collision is glancing, we have that $y_i(t_0) = y_{i+1}(t_0)$ and $y_i'(t_0) = y_{i+1}'(t_0)$. By the above equalities, $y'_i(t_0) = y'_{i+1}(t_0)$ implies
            \begin{align}
                y_i'(0) + \frac{k - m_{i+1}}{2} t_0 &= y_{i+1}'(0) + \frac{k + m_{i}}{2} t_0,\nonumber \\
                y_i'(0) - y_{i+1}'(0) &= \frac{m_i + m_{i+1}}{2} t_0 \label{velocity_difference_at_0}
            \end{align}
            and $y_i(t_0) = y_{i+1}(t_0)$ implies
            \begin{align}
                y_i(0) + y_i'(0)t_0 + \frac{k - m_{i+1}}{4} t_0^2 &= y_{i+1}(0) + y_{i+1}'(0)t_0 + \frac{k + m_{i}}{4} t_0^2 \nonumber\\
                \big( y'_i(0) - y'_{i+1}(0) \big) t_0 - \frac{m_i + m_{i+1}}{4}t_0^2 &= y_{i+1}(0) - y_i(0). \label{position_difference_at_0}
            \end{align}
            By substituting \eqref{velocity_difference_at_0} into \eqref{position_difference_at_0}, we get
            \begin{align*}
                \frac{m_i + m_{i+1}}{4}t_0^2 &= y_{i+1}(0) - y_i(0) 
                \\
                t_0 &= 2\sqrt{\frac{y_{i+1}(0) - y_i(0)}{m_i + m_{i+1}}} = t_c.
            \end{align*}
            Substituting this value of $t_0$ into \eqref{velocity_difference_at_0}, we get
            \begin{align*}
                y_i'(0) - y_{i+1}'(0) &= \frac{m_i + m_{i+1}}{2} \cdot 2\sqrt{\frac{y_{i+1}(0) - y_i(0)}{m_i + m_{i+1}}} = \sqrt{\big( y_{i+1}(0) - y_i(0) \big) \left(m_i + m_{i+1} \right)}.
            \end{align*}
            Thus, we have proven that \emph{(i)} implies \emph{(ii)} and \emph{(iii)}. Now, assume \eqref{veloc-perfect} holds.
            We want to show that the particles collide glancingly at $t_c$, that is, $y_i(t_c) - y_{i + 1}(t_c) = 0$, $y_i'(t_c) - y_{i + 1}'(t_c) = 0$. From \eqref{pos__reference_i} and \eqref{pos__reference_i+1}, 
            \begin{align*}
                y_i(t_c) - y_{i + 1}(t_c) 
                =& -\big( y_{i+1}(0) - y_i(0) \big) + \big( y'_i(0) - y_{i+1}'(0)\big) t_c - \left( \frac{m_i + m_{i+1}}{4} \right) t_c^2 \\
                =& -\big( y_{i+1}(0) - y_i(0) \big) + \left( \sqrt{\big( y_{i+1}(0) - y_i(0) \big) \left(m_i + m_{i+1} \right)} \right) \left( 2\sqrt{\frac{y_{i+1}(0) - y_i(0)}{m_i + m_{i+1}}} \right)  \\ 
                &\quad\quad - \left( \frac{m_i + m_{i+1}}{4} \right) \left( 4 \frac{y_{i+1}(0) - y_i(0)}{m_i + m_{i+1}} \right) \nonumber \\
                =& -\big( y_{i+1}(0) - y_i(0) \big) + 2 \big( y_{i+1}(0) - y_{i}(0) \big) - \big( y_{i+1}(0) - y_{i}(0) \big)\ = \ 0,
            \end{align*}
            and from \eqref{vel__reference_i} and \eqref{vel__reference_i+1} we have
            \begin{align*}
                y_i'(t_c) - y_{i + 1}'(t_c) 
                =& \left( y_i'(0) - y_{i+1}'(0) \right) - \frac{m_{i+1} + m_i}{4} t_c \\
                =& \sqrt{\big( y_{i+1}(0) - y_i(0) \big) \left(m_i + m_{i+1} \right)} - \frac{m_{i+1} + m_i}{2} 2\sqrt{\frac{y_{i+1}(0) - y_i(0)}{m_i + m_{i+1}}} 
                = \ 0.
            \end{align*}
            Thus \emph{(ii)} implies \emph{(i)}. Finally, we will prove that \emph{(iii)} implies \emph{(ii)}, and therefore \emph{(i)} as well by the above. Given the assumption that $y_i(t_c) = y_{i+1}(t_c)$ and \eqref{pos__reference_i}, \eqref{pos__reference_i+1}, we argue as follows.
            \begin{align*}
                y_i(0)& + y_i'(0)t_c + \frac{k - m_{i+1}}{4} t_c^2 = y_{i+1}(0) + y_{i+1}'(0) t_c + \frac{k + m_{i}}{4} t_c^2 
                \\ 
                &\Rightarrow \left (y_i'(0) - y_{i+1}'(0) \right) t_c = y_{i+1}(0) - y_i(0) + (m_i + m_{i+1}) \frac{y_{i+1}(0)-y_i(0)}{m_i+m_{i+1}} \\
                &\Rightarrow y_i'(0) - y_{i+1}'(0) = \frac{2\big(y_{i+1}(0) - y_i(0)\big)}{2\sqrt{\frac{y_{i+1}(0)-y_i(0)}{m_i+m_{i+1}}}} 
                =\sqrt{\big(y_{i+1}(0)-y_i(0)\big)(m_i+m_{i+1})}. \qedhere
            \end{align*}
        \end{proof}
        The following lemma handles a specific situation relevant to Theorem \ref{perfect_e_and_u}, and is the final preliminary to its proof.
        \begin{lemma}\label{Negative Relative Velocity}
            Let $\rho_0$ be a distribution of finitely many particles with masses $m_1,\ldots,m_n$ and trajectories $y_1\leq\ldots\leq y_n$. Assume $(\rho, v)$ is a perfect solution originating at the distribution $\rho_0$. If two adjacent particles $(m_i, y_i(0))$ and $(m_{i+1}, y_{i+1}(0))$ with $y_i(0) < y_{i+1}(0)$  do not engage in any collisions before \mbox{$t_c=2\sqrt{\frac{y_{i+1}(0)-y_i(0)}{m_i+m_{i+1}}}$} and do not collide with each other at $t_c$, then the particles are moving away from each other at $t_c$, i.e., $y_{i+1}'(t_c)-y_i'(t_c) > 0$.
        \end{lemma}
        \begin{proof}
            In order to not have collided by $t_c$, the particles must have less initial relative velocity than they would need to collide at $t_c$ barring earlier collisions with other particles. Thus, by Lemma \ref{perfect-adjacent-particles},
            \begin{align} \label{rel_velocity_bound}
                y'_i(0) - y'_{i+1}(0) < \sqrt{ \big( y_{i+1}(0) - y_i(0) \big) ( m_i + m_{i+1})}.
            \end{align}
            Now, let $k = \sum_{j < i}{m_j} - \sum_{j > i + 1}{m_j}$. Then, for $t < t_c$, $y_i''(t)$ is given by $\frac{k-m_{i+1}}{2}$, and $y_{i+1}''(t)$ is given by $\frac{k+m_{i}}{2}$. Thus, by the derivative of \eqref{position_double_integral}, we have
            \begin{align*}
                y'_i(t) = y'_i(0) + \frac{k - m_{i+1}}{2} t,\quad  
                y'_{i+1}(t) = y'_{i+1}(0) + \frac{k + m_i}{2} t
            \end{align*}
            for $t<t_c$. By continuity of velocities (a consequence of Remark \ref{lipschitz}), these statements hold for $t=t_c$ as well. From this fact, and \eqref{rel_velocity_bound}, the proof is completed:
            \begin{align*}
                y'_{i+1}(t_c) - y'_i(t_c) 
                &= \frac{m_i + m_{i+1}}{2} t_c - \big( y'_i(0) - y'_{i+1}(0) \big) \\
                &= \left( m_i + m_{i+1} \right) \sqrt{\frac{y'_{i+1}(0) - y'_{i}(0)}{m_i + m_{i+1}}} - \big( y'_i(0) - y'_{i+1}(0) \big) \\
                & = \sqrt{\big( y'_{i+1}(0) - y'_{i}(0) \big) \big( m_i + m_{i+1} \big)} - \big( y'_i(0) - y'_{i+1}(0) \big) 
                > 0. \qedhere
            \end{align*}
        \end{proof}
            We can now prove the main result of the section, by showing that in a perfect solution the pair of particles which minimize $t_c$ given in Lemma \ref{perfect-adjacent-particles} must be the first pair to collide. This allows us to induct on the number of particles by applying results from Section \ref{particle-combining-section} to this first collision.
        \begin{theorem} \label{perfect_e_and_u}
            Given an initial mass distribution $\rho_0\in\mathcal{P}(\mathbb{R})$ of finitely many particles with $0$ center of mass and $0$ total momentum, there is a unique initial velocity $v_0\in L^2(\rho_0)$ which yields a perfect DSPS.
        \end{theorem}
        \begin{proof}
            We will induct on the number of particles. The base case of one particle is trivial; the only initial condition to check is a particle with $m=1$ at $y=0$, and $v=0$ is the desired unique velocity, as any other would violate $0$ center of mass and $0$ total momentum. 

            Now assume for induction that every solution with an initial mass distribution of $k-1$ particles has a uniquely determined initial velocity distribution which produces a perfect solution. Consider an arbitrary solution with an initial mass distribution of $k$ particles, with initial  masses and positions given by $m_1,\ldots, m_k$ and $y_1(0)\leq\ldots\leq y_k(0)$. The initial velocity $v_0$ will be determined $\rho_0$ almost everywhere by $v_0(y_i(0))=y_i'(0)$. We shall call this arbitrary distribution ``distribution A.'' By Lemma \ref{perfect-adjacent-particles}, $t_c=2\sqrt{\frac{y_{i+1}(0)-y_i(0)}{m_i+m_{i+1}}}$ is the time it takes for two adjacent particles to collide glancingly without any interfering collisions occurring. Since the first collision in the solution must involve two adjacent particles, and no earlier collisions can interfere with the first collision, this equation must hold for the first collision if the solution is to be perfect. Now consider the pair of adjacent particles for which $t_c$ is minimized. The pair of particles which could potentially glancingly collide in the least time will be the $j$th and $(j+1)$th, where $j$ minimizes the quantity
            \begin{align*}
                \frac{y_{i+1}(0)-y_i(0)}{m_i+m_{i+1}}
            \end{align*}
            for $1\le i \le k-1$. Since this pair minimizes $t_c$, which must give the time of the first collision in a perfect solution, the minimum time at which any collision can occur in a perfect solution is
            \begin{align*}
                t_m=2\sqrt{\frac{y_{j+1}(0)-y_j(0)}{m_j+m_{j+1}}}.
            \end{align*}
            We claim that, if a solution is to be perfect, this earliest potential collision time must be an actual collision time;  $y_j(t_m)=y_{j+1}(t_m)$. Note that in the case of multiple $i$ yielding a minimum, multiple collisions will happen at time $t_m$, and we may pick any of the minimizing $i$.
            
            These particles cannot have collided glancingly before $t_m$ due to Lemma \ref{perfect-adjacent-particles}. We will show that they cannot collide glancingly at $t > t_m$ in a perfect solution, leaving as the only option collision at precisely $t_m$. Assume for contradiction that their collision occurs after $t_m$ in a perfect solution. Then we must have $y_j(t_m) < y_{j+1}(t_m)$. By Lemma \ref{Negative Relative Velocity}, this implies that $y_j'(t_m) < y_{j+1}'(t_m)$. Also, $y_j''(t) \le y_{j+1}''(t)$ for all $t \ge t_m$, as the particle on the left will always have more mass to its right and less mass to its left. Viewing the solution as having initial conditions at $t_m$ (this uses the obvious ``semigroup'' property of the DSPS), the three above inequalities become $y_j(0) < y_{j+1}(0)$, $y_j'(0) < y_{j+1}'(0)$, and $y_j''(t) \le y_{j+1}''(t)$ for all $t$. With these inequalities, it is clear using \eqref{position_double_integral} that $y_j(t) < y_{j+1}(t)$ for all $t$. Translating back through the semigroup property, this tells us that $y_j(t) < y_{j+1}(t)$ for all $t \ge t_m$. That is, the particles do not collide after $t_m$. This is a contradiction, so our assumption that these particles collide after $t_m$ in a perfect solution must have been false. Thus it is impossible in a perfect solution for these particles to collide after $t_m$. Since the particles also cannot collide before $t_m$ in a perfect solution, collision at $t_m$ is necessary if the solution is to be perfect; $y_j(t_m)=y_{j+1}(t_m)$. Since no interfering collisions can happen before this time, we also get from Lemma \ref{perfect-adjacent-particles} that
            \begin{align}\label{Velocity Condition}
                y_j'(0)-y_{j+1}'(0)= \sqrt{\big(y_{j+1}(0)-y_j(0)\big)\big(m_j+m_{j+1}\big)}.
            \end{align} 
            Now, we can construct a dependent solution with initial $k-1$ particle distribution ``distribution B'' given by $m_i^*$ and $y_i^*(0)$, where $m_i^* = m_i$ and $y_i^*(0) = y_i(0)$ for $i<j$, $m_j^* = \frac{m_j+m_{j+1}}{2}$ and $y_j^*(0) = \frac{m_jy_j(0)+m_{j+1}y_{j+1}(0)}{m_j+m_{j+1}}$, and $m_i^* = m_{i+1}$ and $y_i^*(0) = y_{i+1}(0)$ for $i>j$. By our inductive hypothesis, there is a unique set of initial velocities $(y_i^*)'(0)$ which will yield a perfect solution when applied to distribution B. 

            We will now proceed to validating our existence and uniqueness claim for distribution A. By Lemma \ref{PC1}., distributions A and B will yield identical solutions for all $t \ge t_m$ provided that $y_i'(0) = (y_i^*)'(0)$ for $i<j$, $\frac{m_jy_j'(0) + m_{j+1}y_{j+1}'(0)}{m_j+m_{j+1}} = (y_j^*)'(0)$, and $y_i'(0) = (y_{i-1}^*)'(0)$ for $i>j+1$. As we have discovered, \eqref{Velocity Condition} is necessary for distribution A to yield a perfect solution. This and $\frac{m_jy_j'(0) + m_{j+1}y_{j+1}'(0)}{m_j+m_{j+1}} = (y_j^*)'(0)$ are both true only for $y_j'(0) = {y_j^*}'(0) + m_{j+1}\sqrt{\frac{y_{j+1}(0)-y_j(0)}{m_j+m_{j+1}}}$ and $y_{j+1}'(0) = (y_j^*)'(0) - m_{j}\sqrt{\frac{y_{j+1}(0)-y_j(0)}{m_j+m_{j+1}}}$. If we assign distribution A the set of velocities given by $y_i'(0) = (y_i^*)'(0)$ for $i<j$, $y_j'(0)$ and $y_{j+1}'(0)$ as specified above, and $y_i'(0) = (y_{i-1}^*)'(0)$ for $i>j+1$, we must have that the first collision occurs glancingly between particles $j$ and $j+1$ at $t_m$ as \eqref{Velocity Condition} is satisfied. Thus no collisions occur before $t_m$, and only glancing collisions occur at or after $t_m$ since all of the conditions for Lemma \ref{PC1} to apply are satisfied, and thus the solution is identical to the solution where the initial velocities $(y_i^*)'(0)$ are applied to distribution B, which is perfect by assumption and thus all collisions are glancing. 
            
            Now to handle uniqueness, assume for contradiction that some different set of initial velocities $\tilde{y}_i'(0)$ applied to distribution A yields a perfect solution. We define an analog of this to distribution B: let $(\tilde{y}_i^*)'(0) = \tilde{y}_i'(0)$ for $i<j$, $(\tilde{y}_j^*)'(0) = \frac{\tilde{y}_j'(0) + \tilde{y}_{j+1}'(0)}{m_j+m_{j+1}}$, and $(\tilde{y}_i^*)'(0) = \tilde{y}_{i+1}'(0)$ for $i>j$. Note that since \eqref{Velocity Condition} must still be satisfied, if $y_j'(0) \neq \tilde{y}_j'(0)$ or $y_{j+1}'(0) \neq \tilde{y}_{j+1}'(0)$ then we must also have $(y_j^*)'(0) \neq (\tilde{y}_j^*)'(0)$. Thus, the fact that $\tilde{y}_i'(0)$ is different from $y_i'(0)$ for some $i$ implies that $(\tilde{y}_i^*)'(0)$ is different from $(y_i^*)'(0)$ for some $i$. However, given the way we have constructed $(\tilde{y}_i^*)'(0)$, Lemma \ref{PC1} applies, and thus the solution given by applying initial velocities $\tilde{y}_i'(0)$ to distribution A is identical for all $t \ge t_m$ to the solution given by applying initial velocities $(\tilde{y}_i^*)'(0)$ to distribution B. Since this solution has no collisions before $t=t_m$, and is identical to the former (perfect) solution for $t \ge t_m$, all collisions occurring at or after $t_m$ (that is, all collisions in the solution) are glancing. Thus applying initial velocities $(\tilde{y}_i^*)'(0)$ to distribution B yields a perfect solution. However, as previously stated, $(\tilde{y}_i^*)'(0)$ is different from $(y_i^*)'(0)$. Given that both initial velocity distributions yield a perfect solution when applied to distribution B, this violates our induction assumption that distribution B had a unique velocity distribution yielding a perfect solution. Thus our assumption that some different set of initial velocities applied to distribution A yields a perfect solution must have been false. Thus under the assumption that every $k-1$ particle initial distribution has a unique velocity distribution yielding a perfect solution, we have proven that an arbitrary $k$ particle initial distribution has a unique velocity distribution yielding a perfect solution. By induction our claim is validated for all $k$, that is, for all finite particle distributions.
        \end{proof}
        The next section uses insight offered by the properties of perfect solutions to provide a sharp bound on the time of collapse into the equilibrium in terms of the size of the support of the initial distribution. First, we show that perfect DSPS cannot collapse in finite time if the initial support is unbounded.
    
        \begin{theorem}
            A perfect DSPS with unbounded initial support can only approach the equilibrium asymptotically (i.e., there is no finite-time collapse). 
        \end{theorem}
    
        \begin{proof}
            We will prove this by creating a bound for the time it takes a particle to come to equilibrium in terms of the distance that particle starts from origin. In the case of unbounded support, we can find particles arbitrarily far from origin, which will provide an arbitrarily high bound on the time it takes for such a system to collapse, thus excluding the possibility of finite time collapse.
            
            We accomplish this by finding two preliminary bounds. Firstly, note that if the system were to reach equilibrium, then every particle must take a path from its starting position to $y=0$, eventually reaching it with zero velocity. This is the basis for our two bounds: firstly, that the first time a given particle crosses $y=0$ must be less than or equal to the time it takes for that particle to reach $y=0$ and stay there for all time. And secondly, the first time a particle reaches zero velocity must be less than or equal to the time it takes for this particle to reach zero velocity at $y=0$ and stay there for all time. We will assume that the initial distribution is unbounded on the right; the proof would follow symmetrically if it were so on the left.
            
            \textit{Bound 1:} Consider a particle whose location is given by $y(t)$. We will assume without loss of generality that $y(0) > 0$. We aim to place a lower bound on the time which this particle will take to first cross $y=0$. 
            Because $y'(t)$ is Lipschitz continuous by Remark \ref{lipschitz}\footnote{Since the Hamiltonian of a perfect DSPS must be constant, (Definition \ref{perfect-generalization}) there are no non-glancing collisions. (This can be argued similarly to the proof of Proposition \ref{KE-discontinuity}, just noting that if a trajectory is not left- or right-differentiable, the Hamiltonian is not constant.) Accordingly, we shall make use of Remark \ref{lipschitz} throughout this proof.} and $y''(t)>-1/2$ for a.e. $t$ we integrate twice to get
            \begin{align*}
                y(t) > -\frac{1}{4}t^2 + ty'(0) + y(0)\ \mbox{ for all }\ t>0.
            \end{align*}
            The positive root of the quadratic function on the right will now be our time bound for the actual time of equilibrium, denoted by $t_E$, meaning our first bound is
            \begin{align*}
                t_E > 2\left(y'(0) + \sqrt{y'(0)^2+y(0)}\right).
            \end{align*}
            
            \textit{Bound 2:} Given again a particle with $y(0) > 0$, we now aim to bound the time it takes for this particle to first reach $y'(t)=0$. We first use the Lipschitz continuity of $y'(t)$ again, along with $|y''(t)|<1/2$ a.e., to say
            \begin{align*}
                \left|y'(t)-y'(0)\right|&<\frac{1}{2}t\ \mbox{ for all }\ t>0.
            \end{align*}
            Setting $y'(t) = 0$, we find that $2\left|y'(0)\right| <t$. This means the time at which the particle first reaches zero velocity is greater than $2\left|y'(0)\right|$, meaning that our second bound is $t_E > 2\left|y'(0)\right|$.
            
            So, we have simultaneously achieved the bounds $t_E > 2\left(y'(0) + \sqrt{y'(0)^2+y(0)}\right)$ and $t_E > 2\left|y'(0)\right| > -y'(0)$. Adding these bounds gives us $t_E>\sqrt{y'(0)^2+y(0)}$, so $t_E>\sqrt{y(0)}$. Since we have particles with initial positions arbitrarily far from the origin, we conclude $t_E=\infty$. 
            \end{proof}
\section{Finite Time Collapse to Equilibrium}
    \subsection{A Quadratic Bound on Compactly Supported Solutions}\label{quadratic-section}
    
        Since perfect solutions are energetically ``most efficient,'' one would expect them to be the slowest to converge. That is to say, one would expect the system with least energy to be the one that converges in the longest amount of time. Any solution with the same initial distribution (so the same initial potential) given more kinetic energy \textit{that still converges} would be expected to converge faster. This intuition leads to the following bound:\footnote{See section \ref{conjectures-section} for discussion on these results in the context of GSPS.}
    
        \begin{theorem}\label{quad-envelope}
            Let $(\rho,v)$ be a finite DSPS to \eqref{PEP}, with $\rho_0$ compactly supported and with 0 center of mass and 0 total momentum. Let $y_L(t)=\inf\spt(\rho_t)$ and $y_R(t)=\sup\spt(\rho_t)$. Then $(\rho,v)$ goes to equilibrium if and only if $\spt \rho_t \subseteq [f_L(t),f_R(t)]$ for all t, where $f_L(t)=y_L(0)f(t)$ and $f_R(t)=y_R(0)f(t)$ for
            \begin{align}\label{quad-envelope1}
                f(t)&=\frac{\bigg[\left(2\sqrt{y_R(0)-y_L(0)}-t\right)_+\bigg]^2}{4\big(y_R(0)-y_L(0)\big)}.
            \end{align} 
        \end{theorem}

        \begin{figure}
            \centering
            \begin{tikzpicture}[scale=1]
                \draw (-2-3,0)--(2-3,0);
                \draw (0-3,-0.25)--(0-3,3.5);
                
                \draw (0.25-3,3.25) node {$t$};
                \draw (1.7-3,0.25) node {$y$};
            
                \draw plot[smooth, domain=-1-3:-0.00276972-3] ({\x}, {-2*sqrt(-1.3333*(\x+3)+5.7777)++5.3333});
                \draw[fill] (-1-3,0) circle [radius=2.3*0.5pt];

                \draw plot[smooth, domain=0.00714432-3:1-3] ({\x}, {-2*sqrt(3*(\x+3)+33)+12});
                \draw[fill] (1-3,0) circle [radius=2.3*0.28867pt];

                \draw plot[smooth, domain=0.00714432-3:0.25-3] ({\x}, {-2*sqrt(6*(\x+3)+15/2)+6});
                \draw[fill] (0.25-3,0) circle [radius=2.3*0.816496pt];

                \draw plot[smooth, domain=-0.0027697-3:0.010125-3] ({\x}, {-2*sqrt(4*(\x+3)+7.50664)+6});
                
                \draw[fill] (0-3,0.525) circle [radius=2.3*1pt];

                \draw [dashed] plot[smooth, domain=-1-3:0-3] ({\x}, {2.8284271*(1-sqrt(-(\x+3)))});
                \draw [dashed] plot[smooth, domain=0-3:1-3] ({\x}, {2.8284271*(1-sqrt(\x+3))});

                
                \draw (-2+3,0)--(2+3,0);
                \draw (0+3,-0.25)--(0+3,3.5);

                \draw (0.25+3,3.375) node {$t$};
                \draw (1.7+3,0.25) node {$y$};

                \draw[fill] (-1+3,0) circle [radius=1.5pt];
                \draw[fill] (-0.75+3,3.335) circle [radius=1.5pt];

                \draw[fill] (1+3,0) circle [radius=1.5pt];
                \draw[fill] (0.6+3,2.894427) circle [radius=1.5pt];

                \draw plot[smooth, domain=-1+3:-0.5+3] ({\x}, {-2*sqrt(-2*(\x-3)-1)+2});
                \draw plot[smooth, domain=-0.75+3:-0.5+3] ({\x}, {2*sqrt(-2*(\x-3)-1)+1.95-0.0225});

                \draw plot[smooth, domain=0.5+3:1+3] ({\x}, {-2*sqrt(2*(\x-3)-1)+2});
                \draw plot[smooth, domain=0.5+3:0.6+3] ({\x}, {2*sqrt(2*(\x-3)-1)+1.955});

                \draw [dashed] plot[smooth, domain=-1+3:0+3] ({\x}, {2.8284271*(1-sqrt(-(\x-3)))});
                \draw [dashed] plot[smooth, domain=0+3:1+3] ({\x}, {2.8284271*(1-sqrt(\x-3))});

                \draw[white,fill=white] (2.4,2.75) circle (0.222);
                
                \draw(-4,-0.2)--(-4,0.2);
                \draw (-4,-0.35) node {$\scriptstyle y_L(0)$};
                \draw(-2,-0.2)--(-2,0.2);
                \draw (-2,-0.35) node {$\scriptstyle y_R(0)$};
                \draw(2,-0.2)--(2,0.2);
                \draw (2,-0.35) node {$\scriptstyle y_L(0)$};
                \draw (4,-0.35) node {$\scriptstyle y_R(0)$};
                \draw(4,-0.2)--(4,0.2);
                
                \draw(3-0.2,2.82)--(3+0.2,2.82);
                
                \draw(-3-0.2,2.82)--(-3+0.2,2.82);
                \draw (-4.2,2.8) node {$\scriptstyle2\sqrt{y_R(0)-y_L(0)}$};

                \draw(3-0.2,2.82)--(3+0.2,2.82);
                \draw (1.8,2.75) node {$\scriptstyle2\sqrt{y_R(0)-y_L(0)}$};

            \end{tikzpicture}
            \caption{The quadratic bound described in Theorem \ref{quad-envelope} is shown in dashed lines. Left: the solution stays inside the bound and converges to equilibrium. Right: The solution exits the bound and does not converge to equilibrium.}
            \label{fig:quadratic-bound}
        \end{figure}
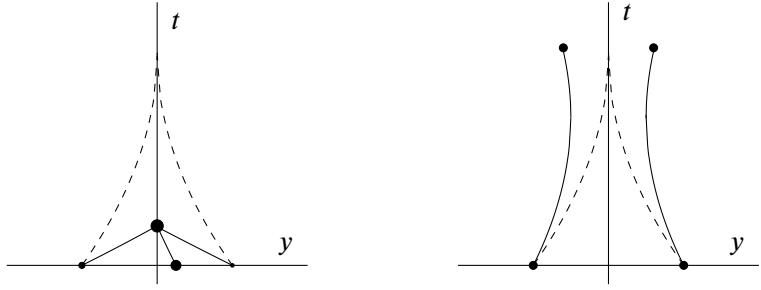 
        
        This bound is sharp; it traces out the trajectory of a two-particle perfect SPS, where the particles begin at $y_L(0)$ and $y_R(0)$, and the mass is distributed among them in order to make the center of mass 0.\footnote{Compare this bound to the parabolic bound on evolution of \emph{any} solution to the \emph{attractive} PEP equations presented in Theorem 3.11 of \cite{BLPTW}. One would expect that the concavity of the bounding parabolas is inverted; interestingly, this bound no longer depends on $\|v_0\|_{L^\infty(\rho_0)}$, as solutions with enough outward velocity will simply not converge to equilibrium.}
    
        \begin{proof}
            Define the critical mass of the leftmost particle, $c_L:=\frac{y_R(0)}{y_R(0)-y_L(0)}$, and the critical mass of the rightmost particle, $c_R:=\frac{-y_L(0)}{y_R(0)-y_L(0)}$. Without loss of generality we shall assume the rightmost particle leaves the envelope described in the theorem. We will draw a contradiction in two cases: one where the rightmost particle gains more mass than a critical mass before coming to $y=0$, and one where it never gains such mass ("gaining mass" intuitively denotes the process of the rightmost particle being redefined to a higher mass particle each time there is a collision involving the previous rightmost).
            
            {\textit{Case 1: The rightmost particle never gains more than critical mass after leaving  the envelope.}}
            
            $y_R(t)$ is continuous, so consider the time T such that $y_R(T)=f_R(T)$, $y'_R(T)>f'_R(T)$ and $y''_R(T)\geq f''_R(T)$. This means we are assuming that the rightmost particle is on the boundary of the envelope with enough velocity to escape the envelope at time $T$ (before reaching equilibrium). Note that since $y_R(T)$ is the position of the rightmost particle, collisions can only increase $y'_R(t)$. Additionally, since  the rightmost particle never gains more than the critical mass we can say that $y''_R(t)\geq f''_R(t)$ at all times, since $f_R''(t)$ is the acceleration of the rightmost particle if it had exactly the critical mass (less mass than this means that more mass is to its left to exert a higher acceleration).
            
            Consider the distance between $y_R(t)$ and $f_R(t)$. We will show that this increases at all times. Let $g(t):=y_R(t)-f_R(t)$. Note that $f_R(t)$ is twice continuously differentiable always, $y_R(t)$ is always right differentiable and has second derivative that is piecewise-constant. 
            
            So, $g(T)=0$ and $g'(t)=y'_R(t)-f'_R(t)$, thus by assumption $g'(T)>0$. Finally, $y''_R(t)-f''_R(t)\geq0$ always, hence $g'(t)$ is always nondecreasing and is positive for $t=T$, so $g'(t)$ is positive for all times $t\geq T$. Thus $y_R(t)>f_R(t)$ for all times $t\geq T$. 
            
    
            $f_R(t)\geq0$ at all times and $y_R(t)>f_R(t)$ for all times $t\geq T$, before which the system had not come to equilibrium, hence $y_R(t)>0$ for all times $t\geq T$. But in order to reach equilibrium, there must be a time $T'$ occurring after $T$ such that $y_R(T')=0$. Hence this solution can never reach equilibrium. 
            

            {\textit{Case 2:  The rightmost particle gains more than critical mass after leaving  the envelope.}}
            
            Assume that the rightmost particle leaves the envelope, and that it has more than the critical mass at some time. There is a time $T$ at which the particle has more than the critical mass, and is simultaneously at or outside the envelope.\footnote{The only way that this could be contradicted is if the particle leaves the envelope with less than the critical mass, then reenters, then gains more than the critical mass; this is impossible by the reasoning in case 1.} At this time $T$, we have $y_R(T)\ge f_R(T)$ and $m_R>\frac{-y_L(0)}{y_R(0)-y_L(0)}$, and therefore the total amount of mass to the left of $m_R$ is less than $1-m_R=\frac{y_R(0)}{y_R(0)-y_L(0)}$. We need the center of mass of the system to be $0$ at all times. Assume for contradiction that the leftmost mass is within the envelope on the left side. Then the center of mass would be minimized if all of the mass not in $m_R$ were concentrated at the leftmost point. That is, $y_L(T)=f_L(T)$ and $m_L=1-m_R< \frac{y_R(0)}{y_R(0)-y_L(0)}$. We can bound the center of mass like so: 
            \begin{align*}
                y_{cm}&> f_L(T)\frac{y_R(0)}{y_R(0)-y_L(0)} + f_R(T)\frac{-y_L(0)}{y_R(0)-y_L(0)}=f(t)\bigg[\frac{y_L(0)y_R(0)}{y_R(0)-y_L(0)} + \frac{-y_L(0)y_R(0)}{y_R(0)-y_L(0)}\bigg]=0.
            \end{align*}
            This is a contradiction with the fact that our center of mass is equal to 0, so our leftmost mass must have left the envelope on the left side. Also, the left mass never gains more than its critical mass without reaching equilibrium, as this would contradict $m_L+m_R \le 1$. Thus by the reflection of case 1, the leftmost particle will never reach $y=0$, hence equilibrium cannot be reached.
        \end{proof}
    
        \begin{corollary}\label{time-bound-corollary}
            Any finite DSPS to \eqref{PEP} that reaches equilibrium must do so by time $2\sqrt{y_R(0)-y_L(0)}$, the time at which the bounding parabolas intersect at $y=0$.
        \end{corollary}
        
        This two-particle solution which bounds other solutions can be thought of as the solution with a given center of mass and smallest closed interval containing the support which maximizes the variance. Noting that $W_2^2(\rho_t,\delta_0)$ is precisely the variance, we obtain the following corollary. 
        \begin{corollary}
            Let $(\rho_t,v_t)$ be a finite DSPS to \eqref{PEP} with $0$ center of mass and $0$ total momentum which is initially supported in $[y_L,y_R]$, and which goes to equilibrium. Then, 
            \begin{align}
                W_2(\rho_t,\delta_0)\leq\frac{\sqrt{-y_Ry_L}}{4(y_R-y_L)}\bigg[\big(2\sqrt{y_R-y_L}-t\big)_+\bigg]^2\ \mbox{ for all }\ t\geq0.
            \end{align}
        \end{corollary}
        This bound is sharp for the same reason as the quadratic bound. 
    \subsection{Equilibrium Conditions for Finite Particles}

        This section considers a strategy we term ``$k$-splitting.'' The general idea is to turn an $n$-particle system into a $2$-particle system and determine whether or not that converges, as $2$-particle systems have simple necessary and sufficient conditions in order to converge. 
        
        \begin{definition}
                    The $k$-split of a $n$ particle DSPS is the $2$-particle system in which one particle is placed at the center of mass of the $k$ leftmost particles with their total mass and velocity given by their total momentum, and one is placed at the center of mass of the $(n-k)$ rightmost particles with their total mass and velocity given by their total momentum.

                    That is, given $\rho_0=\sum_i m_i \delta_{y_i(0)}$ with $y_1(0)\leq y_2(0)\leq\ldots\leq y_n(0)$, we construct a two-particle system of $p_L^k$, $p_R^k$ with initial masses, positions, and velocities as follows: 
                    \begin{alignat*}{6}
                        &m^k_L&&=\sum_{j=1}^km_j, \quad&&y_L^k(0)&&=\frac{\sum_{j=1}^ky_j(0)m_j}{m_L^k}, \quad&&y^k_L{'}(0)&&=\frac{\sum_{j=1}^ky_j'(0)m_j}{m_L^k},
                        \intertext{and}
                        &m^k_R&&=\sum_{j=k+1}^n m_j, \quad&&y_R^k(0)&&=\frac{\sum_{j=k+1}^n y_j(0)m_j}{m_R^k}, \quad&&y^k_R{'}(0)&&=\frac{\sum_{j=k+1}^ny_j'(0)m_j}{m_R^k}~.
                    \end{alignat*}
        \end{definition}

            In some sense splitting is a ``coarse'' way of doing things, losing information, particularly on how the acceleration works. One can see that a general $k$-split will not satisfy the hypotheses of the previously stated particle-combining lemmas, and so will not behave precisely as the true system does; we may have $k$-splits which go to equilibrium for a solution which does not. 
    
        \begin{theorem}[Necessary Condition]\label{discrete-necessary}
            Let $\rho_t=\sum_i m_i\delta_{y_i(t)}$ be a finite DSPS with $y_1\leq\ldots\leq y_n$. The solution reaches equilibrium only if one of its $k$-splits reaches equilibrium. That is, it reaches equilibrium only if there exists an $l\in\mathbb{N}$ such that $$\frac{\sum_{i=1}^ly_i'(0)m_i}{\sum_{i=1}^l m_i}-\frac{\sum_{i=l+1}^ny_i'(0)m_i}{\sum_{i=l+1}^n m_i}\geq\sqrt{\frac{\sum_{i=l+1}^ny_i(0)m_i}{\sum_{i=l+1}^n m_i}-\frac{\sum_{i=1}^l y_i(0)m_i}{\sum_{i=1}^l m_i}}.$$
        \end{theorem}

        \begin{proof}
            Consider a solution that converges to equilibrium. Examine the final collision. Assume that it occurs between $r$ particles $\tilde p_1, \ldots , \tilde p_r$ (ordered by position just before collision) at $t=t_c$. Assume that $\tilde p_1$ is formed of $l$ initial particles. These must be the $l$ leftmost. Now, apply Lemma \ref{PC} with $P = \{p_1, \ldots ,p_l\}$ and $t_P=t_Q=t_c$. This yields a solution which is equivalent for $t \ge t_c$ to our original, but has the $l$ leftmost particles combined in the initial condition. Now apply Lemma \ref{PC} to this system, this time taking $P=\{p_{l+1}, \ldots , p_n\}$ and $t_P=t_Q=t_c$. This yields a third solution which is equivalent for $t \ge t_c$ to the second (and therefore the first), but has the $n-l$ rightmost particles in the original solution combined in its initial condition as well as the $l$ leftmost. But this third solution is exactly the $l$-split; it is a two-particle system with one particle at the center of mass of the $l$ leftmost with their total mass and momentum, and one at the center of mass of the $n-l$ rightmost with their total mass and momentum. Furthermore, for $t \ge t_c$, the $l$-split is equivalent to the original system—that is, it is a single particle in equilibrium. Thus the $l$-split converges.
        \end{proof}
        
        \begin{theorem}[Sufficient Condition]\label{discrete-sufficient}
            Let $\rho_t=\sum_i m_i\delta_{y_i(t)}$ be a finite DSPS with $y_1\leq\ldots\leq y_n$. If all of the solution's k-splits reach equilibrium, then the solution itself reaches equilibrium. That is, if $$\frac{\sum_{i=1}^ky_i'(0)m_i}{\sum_{i=1}^k m_i}-\frac{\sum_{i=k+1}^ny_i'(0)m_i}{\sum_{i=k+1}^n m_i}\geq\sqrt{\frac{\sum_{i=k+1}^ny_i(0)m_i}{\sum_{i=k+1}^n m_i}-\frac{\sum_{i=1}^k y_i(0)m_i}{\sum_{i=1}^k m_i}}$$ for every $k\in\{1,\ldots,n-1\}$ then the solution reaches equilibrium. 
        \end{theorem}

        \begin{proof}
        We claim that the $k$-split reaching equilibrium implies that the $k$th and $(k+1)$th particles collide. This is sufficient to complete the proof as if all adjacent pairs of particles collide, all particles must end up as one.

       By the sticky condition, the $k$th and $(k+1)$th particles must collide before any other collisions occur between any of the $k-1$ leftmost particles and any of the $n-k-1$ rightmost particles. If the $k$th and $(k+1)$th particles meet before the $k$-split comes to equilibrium, we are done immediately, so assume without loss of generality that the $k$th and $(k+1)$th particles do not meet before the $k$-split comes to equilibrium. Now assume the $k$-split reaches equilibrium at time $t_c$. 

        Since $y_i(t)\leq y_k(t)$ for every $i\in\{1,\ldots,k\}$ and every t, we have that $\frac{\sum _{i=1}^k y_i(t)m_i}{\sum _{i=1}^k m_i}\leq \frac{y_k(t)\sum _{i=1}^km_i}{\sum _{i=1}^km_i}=y_k(t)$, and by the same argument $y_{k+1}(t)\leq \frac{\sum_{i=k+1}^n y_i(t)m_i}{\sum_{i=k+1}^n m_i}$. Since no collisions happen between any of the leftmost $k$ particles and any of the $n-k$ rightmost particles, Lemma \ref{PC2} applies to both groups. Since $p^*$ of Lemma \ref{PC2} is defined exactly the same as $p_L^k$ and $p_R^k$ for the respective groups, we have the respective equalities $y_L^k(t)=\frac{\sum_{i=1}^ky_i(t)m_i}{\sum_{i=1}^km_i}$ and $y_R^k(t)=\frac{\sum_{i=k+1}^ny_i(t)m_i}{\sum_{i=k+1}^nm_i}$ for $t \le t_c$. 
        
        But given the inequalities we've established above, for $t \le t_c$ we have
        \begin{align*}
            y_L^k(t)=\frac{\sum_{i=1}^ky_i(t)m_i}{\sum_{i=1}^km_i}\leq y_k(t)\leq y_{k+1}(t)\leq \frac{\sum_{i=k+1}^ny_i(t)m_i}{\sum_{i=k+1}^nm_i}=y_R^k(t).
        \end{align*}
        At $t=t_c$, the convergence of the $k$-split means that the $y_L^k(t_c)=y_R^k(t_c)$, and hence all inequalities are equalities, so $y_k(t_c)=y_{k+1}(t_c)$. Thus the $k$th and $(k+1)$th particles have collided by $t_c$.
    \end{proof}

\subsection{The general case}
        We will now extend these results from the discrete case to general probability measures. Let $\mu\in\mathcal{P}_2(\mathbb{R})$ which has at least two points in its support, $v\in L^2(\mu)$ and pick $y\in\mathbb{R}$ such that $0<M^\mu(y):=\mu((-\infty,y])<1$; in other words, $y\in \big(M^{\mu}\big)^{-1}((0,1))$. Define 
        \begin{alignat*}{4}
            &Y_L(y)&&=\frac{1}{\mu((-\infty,y])}\int_{(-\infty,y]}z\mu(dz),\quad &&Y_R(y)&&=\frac{1}{\mu((y,\infty))}\int_{(y,\infty)}z\mu(dz),
            \\
            &V_L(y)&&=\frac{1}{\mu((-\infty,y])}\int_{(-\infty,y]}v(z)\mu(dz),\quad &&V_R(y)&&=\frac{1}{\mu((y,\infty))}\int_{(y,\infty)}v(z)\mu(dz).
        \end{alignat*}
        Let $$\mathcal{E}[\mu,v](y):=V_L(y)-V_R(y)-\sqrt{Y_R(y)-Y_L(y)}.$$ If $f$ is a $\mu$-measurable function, it is convenient to use the averaging (with respect to the weight $\mu$) notation
        $$\intbar_{B}f(y)\mu(dz):=\frac{1}{\mu(B)}\int_{B}z\mu(dz),$$
        for any Borel set $B\subset\mathbb{R}$ with $\mu(B)>0$.
        Denote by $I$ the closed interval $[-1,1]$. 
        \begin{theorem}\label{continuous-sufficient}
        Let $\rho_0\in\mathcal{P}(\mathbb{R})$, $v_0\in C(\mathbb{R})$ such that
        $$\int_I\rho_0(dz)=1,\quad \int_I z\rho_0(dz)=\int_Iv_0(z)\rho_0(dz)=0.$$ Assume 
        \begin{equation}\label{suff-assump}
        \inf_{y\in \big(M^{\rho_0}\big)^{-1}((0,1))}\mathcal{E}[\rho_0,v_0](y)>0.
        \end{equation}
        Then there exists a GSPS to the repulsive pressureless Euler-Poisson system which starts at $(\rho_0,v_0)$ and collapses into the equilibrium $(\delta_0,0)$ in finite time.
        \end{theorem}
        \begin{proof} Without loss of generality, assume there is no $0<a<1$ such that $\rho_0((a,1])=0$. If there is such an $a$, it is apparent how to appropriately modify this proof. As expected, we will construct such a solution as limit of discrete solutions $(\rho^n,v^n)$ which collapse into the equilibrium by a finite time that is independent of $n$. Let $n\geq 2$ be an integer and set $I_{n,k}:=[k/n,(k+1)/n)$ if $-n\leq k\leq n-2$ and $I_{n,n-1}:=[1-1/n,1]$. Next let $k_1<\ldots<k_{l(n)}$ be the set of all integers $k$ such that $-n\leq k\leq n-1$ and $\rho_0(I_{n,k})>0$. By the assumption that no set of the type $(a,1]$ (for some $0<a<1$) has measure zero, we see that $k_{l(n)}=n-1$. Denote $m_{n,j}:=\rho_0(I_{n,k_j})$, $j=1,\ldots,l(n)$.
        For $j=1,\ldots,l(n)$ define 
        $$y^0_{n,j}:=\intbar_{I_{n,k_j}}z\rho_0(dz).$$
         Since $\rho_0(I_{n,k_j})>0$ and $y^0_{n,j}$ is the center of mass of $
        \rho_0$ on $I_{n,k_j}$, we have $M^{\rho_0}(y^0_{n,j})>0$. Furthermore, $M^{\rho_0}(y^0_{n,j})=1$  if and only if $\rho_0$ is supported in $[-1,y^0_{n,j}]$.  Again, by the assumption that no set of the type $(a,1]$ (for some $0<a<1$) has measure zero, it must be that $j=l(n)$, $k_{l(n)}=n-1$ and $y^0_{n,l(n)}=1$, which means $y^0_{n,l(n)}$ cannot be the center of mass of $I_{n,n-1}=[1-1/n,1]$ with respect to $\rho_0$ unless $\rho_0(I_{n,n-1})=\rho_0(\{1\})>0$. Thus,  $y^0_{n,j}\in \big(M^{\rho_0}\big)^{-1}((0,1))$ for all $j=1,\ldots,l(n)-1$ but not necessarily for $j=l(n)$.
        Next we set $$\rho^n_0:=\sum_{j=1}^{l(n)}m_{n,j}\delta_{y^0_{n,j}},$$ pick $\xi\in C^1_c(\mathbb{R})$ and estimate
        $$\bigg|\int\xi d\rho_0^n-\int\xi d\rho_0\bigg|\leq\sum_{j=1}^{l(n)}\int_{I_{n,k_j}}|\xi(z)-\xi(y^0_{n,j})|\rho_0(dz)\leq\frac{1}{n}\|\xi'\|_\infty,$$
        which implies $W_2(\rho_0^n,\rho_0)\rightarrow 0$ as $n\rightarrow\infty$, given that $\rho_0,\ \rho_0^n$ are probability measures supported in the compact interval $I$. For the same reason, this convergence in the 2-Wasserstein distance implies all the conditions imposed in Corollary \ref{existence-general}. To determine the initial velocities we use, as in the proof of said theorem, the fact that $v_0^n\rho_0^n$ is the distributional derivative of $F^n_0\circ M_0^n$, where $M_0^n$ is the right-continuous cumulative distribution function of $\rho_0^n$ and $F^n_0(m):=\int_0^m v_0\circ N_0^n(\omega)d\omega$. By Corollary 2.3 \cite{Tudorascu-2008}, we see that 
        $$v_0^n(y^0_{n,j})=\int_0^1v_0\circ N_0^n((1-s)M_0^n(y^0_{n,j}-)+sM_0^n(y_{n,j}^0))ds=v_0(y_{n,j}^0).$$
        Thus, to check that $(\rho^n_0,v_0^n)$ satisfies the hypothesis of Theorem 4.8, pick $1\leq j\leq l(n)-1$ and compute 
        $$V^n_L(y^0_{n,j})-V_L(y^0_{n,j})=\frac{\sum_{i=1}^j\int_{I_{n,k_i}}[v_0(y^0_{n,i})-v_0(z)]\rho_0(dz)}{\sum_{i=1}^jm_{n,i}}$$
        and use the continuity of $v_0$ on the compact interval $I$ to conclude that for any $\varepsilon>0$ we have $$|V^n_L(y^0_{n,j})-V_L(y^0_{n,j})|\leq \varepsilon\mbox{ for all }j=1,\ldots,l(n)-1\mbox{ and all sufficiently large }n.$$
        Similar conclusions hold for the $V_R,\ V_R^n,\ Y_L,\ Y_R,\ Y_L^n,\ Y_R^n$ terms, which means that for any given $\varepsilon>0$ and sufficiently large $n$ we have 
        $$\big|\mathcal{E}[\rho_0^n,v_0^n](y^0_{n,j})-\mathcal{E}[\rho_0,v_0](y^0_{n,j})\big|\leq\varepsilon\mbox{ for all }j=1,\ldots,l(n)-1.$$
       
        We see now that \eqref{suff-assump} implies that 
        $$\mathcal{E}[\rho_0^n,v_0^n](y^0_{n,j})>0\mbox{ for all sufficiently large }n\mbox{ and for all }j=1,\ldots,l(n)-1.$$
        So now the hypothesis of Theorem \ref{quad-envelope} are satisfied, which means $(\rho^n,v^n)$ collapses in finite time into the dynamic equilibrium $(\delta_{c^n_0t},c^n_0)$, where $c^n_0:=\int v_0^nd\rho_0^n$. Indeed, it follows by construction that $\rho_0^n$ has zero center of mass. Since 
        $$c_0^n=\sum_{j=1}^{l(n)}m_{n,j}v_0(y_{n,j}^0)-\int v_0d\rho_0=\sum_{j=1}^{l(n)}\int_{I_{n,k_j}}[v_0(y^0_{n,j})-v_0(z)]\rho_0(dz),$$
        we deduce, by the uniform continuity of $v_0$ on $I$, that $c_0^n$ converges to zero as $n\rightarrow\infty$. By Corollary \ref{time-bound-corollary} we know that the time to collapse is bounded above by $2\sqrt{2}$ because all the measures $\rho_0^n$ are supported in $I=[-1,1]$. Therefore, it must be that $\rho_t=\delta_0$ and $v_t(0)=0$ for all $t\geq2\sqrt{2}$. 
        \end{proof}
\section{Numerical Examples}
    \begin{figure}[h!]
    \centering
    \begin{subfigure}[b]{0.4\linewidth}
        \includegraphics[width=\linewidth]{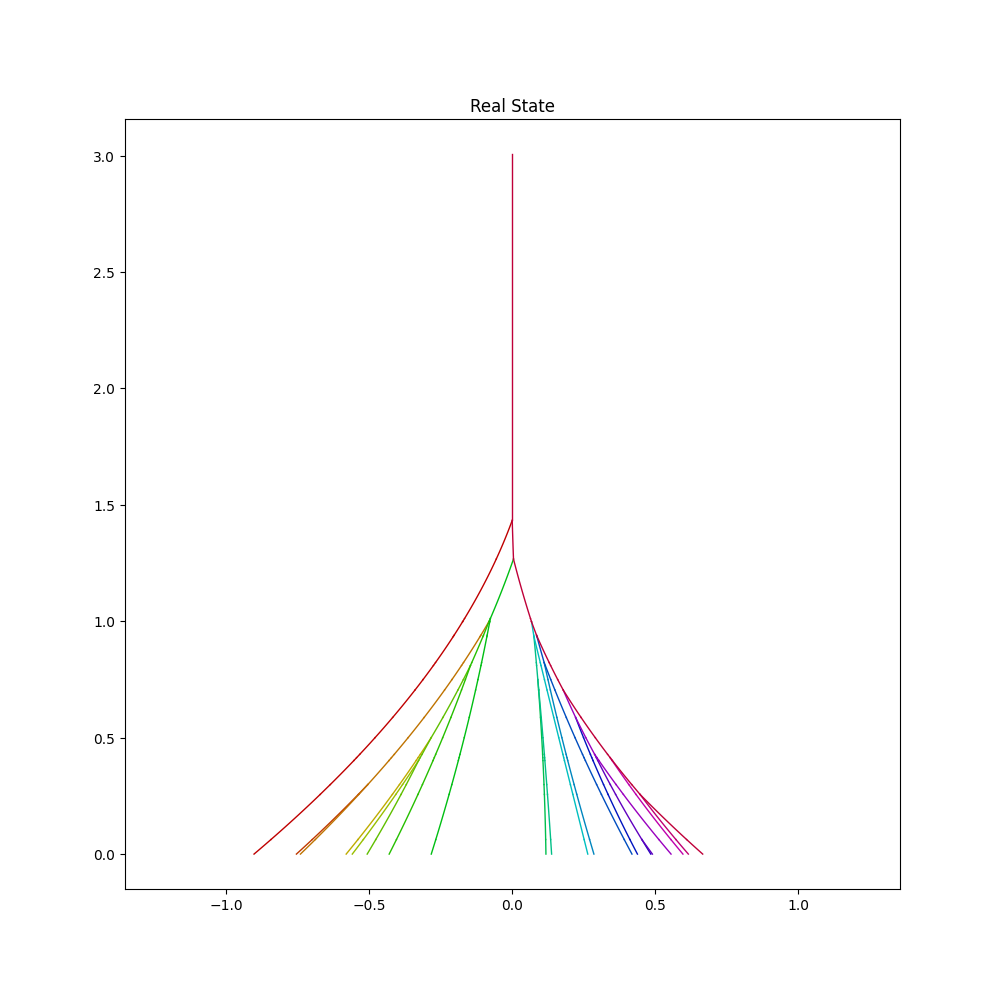}
        \end{subfigure}
        \begin{subfigure}[b]{0.4\linewidth}
        \includegraphics[width=\linewidth]{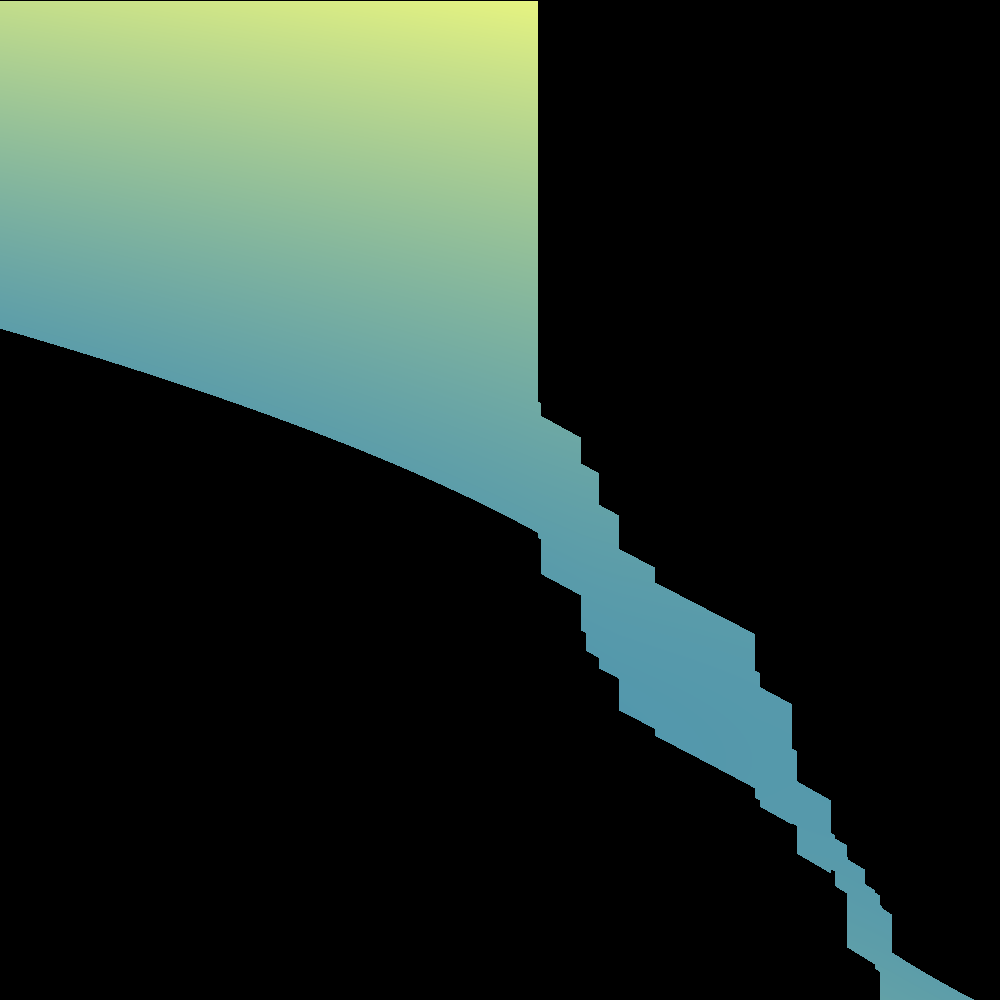}
    \end{subfigure}
    \caption{An example DSPS and the state of the solution as the position and velocity of a particle varies.}
    \label{fig:numerical}
    \end{figure}
    Figure \ref{fig:numerical} demonstrates the sensitivity that the repulsive case has to initial conditions. The image on the left is a random system of particles.\footnote{The particular values for the system of particles can be found in \texttt{region\_data/example.csv} in the GitHub repository. The first row in the CSV are the masses of the particles, the second are the positions, and the third are the velocities. The rest of the rows are cached data and are not important.} The image on the right demonstrates how the convergence of the system changes as we shift the position and velocity of the first particle. The difference from the first particle's original position is on the horizontal axis, and the difference in velocity is on the vertical, both ranging from $-2$ to $2$. Each pixel on the image represents a system that has the first particle of the system on the left shifted by the corresponding values on the horizontal (position) and vertical (velocity) axes. Black represents that the system given does not converge, and any other color represents convergence. A more blue color represents a smaller Hamiltonian, and an orange color represents a larger Hamiltonian. 
    
    The GitHub can be found at \url{https://github.com/CommanderCOOL2005/WVU-Sticky-Particles}. To view this graph in the code, run \texttt{region\_plotter.py} in the git repository. You can click anywhere on the plot to see the evolution of the corresponding system.
\section{Future Steps}\label{conjectures-section}
One would expect that since finite DSPS satisfy the quadratic bound presented in section \ref{quadratic-section}, the solutions which they can approximate, i.e., $\bar{\mathcal{S}}$, would satisfy a similar bound. We pose the following conjecture:
        \begin{conjecture}
            Let $(\rho,v)$ be a GSPS to \eqref{PEP}, with $\rho_0$ compactly supported and with $0$ center of mass and $0$ total momentum. Take $y_L$, $y_R$, and $f$ as given in Theorem \ref{quad-envelope}. Then, there exists $t_0>0$ such that $\rho_{t_0}=\delta_0$ if and only if $\spt\rho_t\subseteq[y_L(0)f(t),y_R(0)f(t)]$ for all $t\in\big(0,\inf(\{2\sqrt{y_R(0)-y_L(0)}\}\cup\{s\geq0|\rho_s=\delta_0\})\big)$. 
        \end{conjecture}
        The reverse implication is clear. As for the direct one, we would need to prove existence of approximating DSPS which converge to equilibrium so that the envelope \eqref{quad-envelope1} applies uniformly; then we can pass to the limit and make the desired inference on the GSPS. Even if one can prove uniform bounds on the Wasserstein distance $W_2(\rho_t^n,\rho_t)$, the fact that $\rho_t=\delta_0$ for sufficiently large $t$ would not suffice to infer that for sufficiently large $n$ the quantity $W_2(\rho_t^n,\rho_t)$ converges to zero as $t\rightarrow\infty$. Thus, it is not clear how to apply Theorem \ref{quad-envelope} in order to prove the above conjecture.
        
        Because solutions in $\bar{\mathcal{S}}$ can collapse to a point mass and subsequently expand, it is reasonable as well to expect that such solutions will also \emph{expand} with at most this quadratic rate. (Consider a solution like Example \ref{bad-approximation}.) Of course, if $\rho_{t_0}=\delta_0$, $\rho_t\neq\delta_0$ for $t>t_0$ would violate the sticky condition, so if $(\rho,v)\in\mathcal{S}$, Theorem \ref{quad-envelope} would apply directly to $(\rho,v)$, as well as both corollaries. We also would have that GSPS to \eqref{PEP} which go to equilibrium asymptotically have unbounded support for all $t \ge 0$.

        Also to be resolved is the question of characterization of asymptotic equilibrium: 
        \begin{question}
            If $(\rho,v)$ is a GSPS to \eqref{PEP} with $W_2(\rho_t,\delta_0)\to0$ as $t\to\infty$, do we also have $\|v_t\|_{L^2(\rho_t)}^2\to0$?
        \end{question} 
 Given the nonincreasing behavior of the Hamiltonian, we do not expect rapid oscillations around $\delta_0$; in fact, if the GSPS collapses to $\delta_0$ in finite time, then the velocity trivially vanishes from the collapse time forward.
 
\section{Acknowledgements}
This work was partially supported by the National Science Foundation under Grant Number DMS-2349040 (PI: C. Tsikkou).
	Any opinions, findings, and conclusions or recommendations expressed in this
material are those of the authors and do not necessarily reflect the views of the National Science Foundation.


\end{document}